\documentclass[]{elsarticle}

\usepackage{lineno,hyperref}
\usepackage[english]{babel}
\usepackage{amssymb}
\usepackage{multicol}
\usepackage{amsfonts}
\usepackage{algpseudocode} 
\usepackage{algorithm} 
\usepackage{graphicx}
\usepackage{caption}
\usepackage{bm}
\usepackage{bbm} 
\usepackage{amsthm} 
\usepackage{amsmath}
\usepackage{color}
\usepackage{picture}
\usepackage{adjustbox}
\usepackage{subcaption} 

\numberwithin{equation}{section}

\modulolinenumbers[5]

\newtheorem{thm}{Theorem}
\newtheorem{lem}{Lemma}
\newtheorem{cor}{Corollary}
\newtheorem{prop}{Proposition}
\newtheorem{rem}{Remark}

\newtheorem{ass}{Assumption}

\definecolor{ForestGreen}{rgb}{0,0.5,0}
 
\def\B{\color{black}}

\def\Ru{\color{black}}
\def\Rd{\color{black}}

\newcommand{\jac}[1]{\ensuremath{\det(D #1)}}

\newcommand{\z}{\phantom{0}}

\def \M{\mathbf{M}}
\def \D{\mathbf{D}}
\def \P{\boldsymbol{\mathcal{M}}}
\def \PCE{\boldsymbol{\mathcal{M}}_{CE}}
\def \Ndof{N_{\text{dof}}}
\def \PMP{\P^{-\frac{1}{2}}\M\P^{-\frac{1}{2}}}
\def \PMPCE{\P_{CE}^{-\frac{1}{2}}\M\P_{CE}^{-\frac{1}{2}}}

\def \Mp{\widehat{\M}}
\def \Mt{\widetilde{\M}}
\def \vett[#1]{\boldsymbol{#1}}
\def \DMD{\D^{-\frac{1}{2}} \widehat{\D}^{\frac{1}{2}} \M \widehat{\D}^{\frac{1}{2}} \D^{-\frac{1}{2}}}

\def \supp{\mathrm{supp}}

\def \Prs{\D^{\frac{1}{2}} \widehat{\D}^{-\frac{1}{2}} \Mp \widehat{\D}^{-\frac{1}{2}} \D^{\frac{1}{2}}}

\def \spann{\mathrm{span}}

\def \NOmega{N_{\mathrm{patch}}}
\def \Npatch{N_{\mathrm{adj}}}
\def \ptca{r}
\def \ptcb{s}
\def \Prsa{\D^{(\ptca)^{\frac{1}{2}}} \widehat{\D}^{(\ptca)^{-\frac{1}{2}}} \Mp^{(\ptca)} \widehat{\D}^{(\ptca)^{-\frac{1}{2}}} \D^{(\ptca)^{\frac{1}{2}}}}

\def \Prad{\P^{-1}_{\text{ad}}}

\def \PadMPad{\P^{-\frac{1}{2}}_{\text{ad}}\M \P^{-\frac{1}{2}}_{\text{ad}}}
\def \ua{{u^{(\ptca)}}}

\begin{document}

\begin{frontmatter}
	
	\title{Easy and efficient preconditioning of  the Isogeometric mass matrix}
	\author[mat]{Gabriele Loli}\ead{gabriele.loli01@universitadipavia.it}
	\author[mat,imati]{Giancarlo Sangalli} \ead{giancarlo.sangalli@unipv.it}
	\author[imati]{Mattia Tani} \ead{mattia.tani@imati.cnr.it}
	
	\address[mat]{Dipartimento di Matematica ``F. Casorati", Universit\`{a} di Pavia, Via A. Ferrata, 5, 27100 Pavia, Italy.}
	
	\address[imati]{Istituto di Matematica Applicata e Tecnologie Informatiche, ``E. Magenes" del CNR, Via A. Ferrata, 1, 27100 Pavia, Italy.}
\begin{abstract}
	This paper deals with the fast solution of linear systems associated
	with the mass  matrix, in the context of isogeometric analysis. We
        propose a preconditioner that is both efficient and easy to implement, based   on a  diagonal-scaled  Kronecker product of univariate parametric
	mass matrices. Its application is faster than a matrix-vector product involving the mass matrix itself. 
	We prove that the condition number of the preconditioned
        matrix converges to 1 as the mesh size is reduced, that is, the preconditioner is asymptotically equivalent to the exact inverse. 
	Moreover, we give numerical evidence of its good behaviour with respect to the spline degree and the (possibly singular) geometry parametrization.
	We also extend the preconditioner to the multipatch case through an Additive Schwarz method.
\end{abstract}
\end{frontmatter}

\section{Introduction}

Isogeometric analysis (IGA) proposed in
\cite{Hughes2005} (see also the book \cite{Cottrell2009}), is   a  computational technique  for solving partial
differential equations  that  uses splines, Non-Uniform Rational
B-splines (NURBS) and other possible
{\Ru generalizations,  both} for the parametrization  of  the computational domain, 
as typically done  in computer aided design, and for the
representation of the unknown  field of  the differential problem.  
Many papers have demonstrated the effective advantage
of isogeometric methods  in various frameworks, see for example the recent
special issue \cite{20171} on the topic. 

The focus of this paper is the solution  of the linear systems
associated with the isogeometric Galerkin  mass matrix, for
arbitrary degree and continuity of the spline approximation.  In
particular, we want to cover the case of  high-degree and
high-continuity spline approximation (the so-called isogeometric
\mbox{$k$-refinement}) whose advantages are explored in, e.g.,
\cite{evans2009n,da2011some,sande2019sharp,takacs2016approximation,bressan2018approximation,SANGALLI2018117}.  
Solving the  mass matrix system  is needed, for example:
\begin{itemize}
	\item in explicit dynamic simulation, that is, when an explicit finite difference
	schemes in time is coupled to an isogeometric discretization
	in space, see e.g. \cite{Hartmann2016} 
	\item in PDE-constrained optimization problem \cite{Dolgov2020Stoll}
	\item when the mass matrix is used as a smoother in a multigrid solver  \cite{hofreither2014mass,hofreither2017robust}
	\item when the mass matrix is used as a preconditioner for the Schur complement of the Stokes problem \cite{elman2014finite} 
	\item in general,  when   evaluations of $L^2$-projections are
	needed, for example in  nearly-incompressible elasticity  with the
	$\bar B$--$\bar F$  method \cite{elguedj2008b},   or in the  mortar
	method for multipatch gluing  \cite{BRIVADIS2015292}, or in other applications
	like fast  simulation of tumor evolution \cite{los2017application}.
\end{itemize}

Due to the condition number of the mass matrix, that grows
exponentially with respect to the spline degree, finding efficient solvers 
is not a trivial task unless we are  in the low degree  case.

One of the first  ideas  that have been explored 
is to use collocation instead of a Galerkin formulation, since in this
case the mass matrix (that is, the B-spline collocation matrix)  is
easier to invert (see  \cite{EVANS2018208,AURICCHIO20122} and the
references therein).

If we stay with the Galerkin formulation, the classical strategy of lumping
and then inverting the mass matrix lacks accuracy and, as a
preconditioner for an iterative solver, lacks robustness with respect
to the spline  degree. There are instead ad hoc constructions of sparse and approximated inverse
of the mass matrix, see for example \cite{tkachuk2015direct},  or
biorthogonal bases, see 
\cite{wunderlich2019biorthogonal}, designed with the aim of keeping  accuracy.
Approximated inverses or preconditioners of the mass matrix often use one  key feature of   multivariate
splines: the tensor-product construction. Indeed $\widehat{\mathbf{M}} $, the
Galerkin mass matrix  on the reference patch  $[0,1]^d$,  is a Kronecker matrix of the form
\begin{equation}\label{eq:Kronecker-structure}
\widehat{\mathbf{M}} =  \widehat{\mathbf{M}}_d\otimes\ldots\otimes \widehat{\mathbf{M}}_1 ,
\end{equation}
where the $\widehat{\mathbf{M}}_i $ are unidimensional parametric mass matrices. Inverting
$\widehat{\mathbf{M}}$ above only requires the inversion of the factors
$\widehat{\mathbf{M}}_i $.
However, on a generic patch, due to the geometry mapping, the 
structure above is lost, that is,  the true  isogeometric mass matrix
$\mathbf{M}$ we are interested in  is not a Kronecker matrix like
$\widehat{\mathbf{M}} $. One could use $\widehat{\mathbf{M}}   $  as a
preconditioner for ${\mathbf{M}} $, but, depending on the geometry
parametrization of the patch, the results are not always satisfactory.
Then,  \cite{GAO201419} developed an extension  
of \eqref{eq:Kronecker-structure} that better  approximate
$\mathbf{M}$ and is suitable for a fast application, see also
\cite{wozniak2017parallel} for its parallel implementation. Another
possibility is to seek for a low-rank approximation of $\mathbf{M}$,
that is, approximate $\mathbf{M}$ as a sum of
Kronecker matrices, see
\cite{mantzaflaris2014matrix,mantzaflaris2017low,hofreither2018black}.
The recent paper \cite{CHAN201822} constructs an approximation of $\mathbf{M}^{-1} $ as $\widehat{\mathbf{M}}
^{-1}\mathbf{M}_{\jac{\vett[F]}^{-1}}\widehat{\mathbf{M}} ^{-1}$, where
$\mathbf{M}_{\jac{\vett[F]}^{-1}}$ is a suitable weighted {\Ru mass matrix.}

In our paper,  we also propose and study a preconditioner $\P$  of the Galerkin
mass matrix. The main feature of our approach is that, compared to
previous results,  it is very easy to implement  but also
extremely efficient and robust.  On a single patch, we define $\P$ as $\Prs$,  where
$\mathbf{D} $ and  $\widehat{\mathbf{D}} $ are the diagonal matrices
made with the diagonals of the true mass ${\mathbf{M}} $ and parametric
mass $\widehat{\mathbf{M}} $, respectively. Therefore, we approximate
${\mathbf{M}} $  by the {\Ru Kronecker matrix} $\widehat{\mathbf{M}} $ combined
with a symmetric diagonal scaling. The computational cost of one
application of the preconditioner is then just $O(p\Ndof)$  FLOPS, while
each matrix-vector multiplication with $\mathbf{M} $ requires $ O(p^d
N_{\text{dof}})$ FLOPS, where $p$ is the spline degree and $
N_{\text{dof}}$ is the number of degrees of freedom. For multipatch domains, we combine the
preconditioner above on each patch with an Additive Schwarz method.
We prove the robustness of the preconditioner with respect to the
mesh size and, in the single patch case, we also show that $
 \kappa(\PMP) \rightarrow 1$ when $h \rightarrow 0$.  Our numerical benchmarks show that the
preconditioned problem behaves well also for large $p$ and even in the
case of typical singular parametrizations of the computational domain,
which is a case not covered by the theory.

The rest of this work is organized as follows: Section
\ref{sec:preliminaries} introduces our notation for B-splines and
isogeometric analysis.  In Section \ref{sec:single_patch} we describe
the proposed  preconditioner on a single patch domain and we prove its
\mbox{$h$-robustness}, while in Section \ref{sec:multipatch} we generalize it
to multipatch domains by means of the Additive Schwarz theory.  We
show how to efficiently apply the  preconditioner and analyze its computational cost in Section \ref{sec:application_cost}.
In Section \ref{sec:tests} we report  numerical results assessing the effectiveness of the proposed preconditioner, its good behaviour with
respect to $p$ and in case of singular parametrizations,  and compare with
the approach of \cite{CHAN201822}.
Concluding remarks 
are
wrapped up in Section \ref{sec:conclusion}.

\section{Preliminaries}\label{sec:preliminaries}
\subsection{B-splines}
Given two positive integers $p$ and $m$, consider an open knot vector $$\Xi :=
\{\xi_1,\ldots, \xi_{m+p+1}\}$$ such that
$$
\xi_1 =\ldots=\xi_{p+1} < \xi_{p+2} \le \ldots \le 
\xi_{m} < \xi_{m+1}=\ldots=\xi_{m+p+1},
$$
where interior repeated knots are allowed with maximum multiplicity $p$. Without loss of generality, we
assume $\xi_1 = 0$ and $\xi_{m+p+1} =1 $.
From the knot vector $\Xi$, B-spline functions of degree $p$ are defined following the well-known Cox-De Boor recursive formula: we start with piecewise constants ($p=0$):
\begin{equation*}
 \widehat b_{i,0}(\zeta) = \left \{
\begin{array}{ll}
1 & \text{if } \xi_i \leq \zeta < \xi_{i+1}, \\
0 & \text{otherwise},
\end{array}
\right.
\end{equation*}
and for $p \ge 1$ the B-spline functions are defined by the recursion
\begin{equation*}
 \widehat b_{i,p}(\zeta) = \frac{\zeta - \xi_i}{\xi_{i+p} - \xi_i}  \widehat b_{i,p-1}(\zeta) + \frac{\xi_{i+p+1} - \zeta}{\xi_{i+p+1} - \xi_{i+1}}  \widehat b_{i+1,p-1}(\zeta),
\end{equation*}
where $0/0 = 0$.
Each B-spline $ \widehat b_{i,p}$ depends only on $p+2$ knots, which are collected in the local knot vector
\begin{equation*}
\Xi_{i,p}:= \{ \xi_{i}, \ldots,\, \xi_{i+p+1}\},
\end{equation*}
is non-negative and supported in the interval $[\xi_i , \xi_{i+p+1}]$. Moreover, these \mbox{B-spline} functions constitute a partition of unity, that is
\begin{align} \label{eq:partition_unity}
	\sum_{i=1}^{m} \widehat b_{i,p} (x)=1, & & \forall x \in (0,1).
\end{align}
The univariate spline space is defined as
\begin{equation*}
\widehat{\mathcal{S}}_{h} = \widehat{\mathcal{S}}_{h}([0,1]) : = \mathrm{span}\{\widehat{b}_{i,p}\}_{i = 1}^m,
\end{equation*}
where $h$ denotes the maximal mesh-size.
For brevity, the degree $p$ is not always reported in the notation.
For more details on B-splines properties see   \cite{Cottrell2009,DeBoor2001}.

Multivariate B-splines are defined from univariate B-splines by tensorization.
  Let $d$ be the space dimension and consider open knot vectors \linebreak ${\Xi_k = \{\xi_{k,1}, \ldots,
\xi_{k,m + p + 1} \}}$ and a set of multi-indices \linebreak ${\mathbf{I}:=\{ \mathbf{i}=(i_1,\ldots, i_d): \, 1 \leq i_l \leq m \}}$. For each multi-index  $\mathbf{i}=(i_1,\ldots, i_d)$, we
introduce the $d$-variate B-spline,
\begin{equation*}
  \label{eq:multivariate-B-splines}
   \widehat B_{\mathbf{i}}(\mathbf{\zeta}) :=  \widehat
    b[\Xi_{i_1,p}](\zeta_1) \ldots  \widehat
    b[\Xi_{i_d,p}](\zeta_d).
\end{equation*}
Observe that, for the sake of simplicity, the knot
vectors are assumed to have the same length and the degree is the same
in all directions. 
The support of each multivariate basis function is
\begin{align*}
	Q'_{\mathbf{i}}:=\supp( \widehat B_{\mathbf{i}})=  \prod_{k=1}^d [\xi_{k,i_k} , \xi_{k,i_k+p+1}].	
\end{align*}
For notational convenience, we define the index set for mesh elements \linebreak
${\mathbf{I}_e := \{ (j_1,\ldots, j_d): \, 1 \leq j_l \leq m+p+1 \}}$,  
\begin{align}\label{def:Q}
		Q_{\mathbf{j}} :=  \prod_{k=1}^d [\xi_{k,j_k} , \xi_{k,j_k+1}], & & \mathbf{j} \in \mathbf{I}_e
\end{align}
and
\begin{align}\label{def:index_set_1}
	\mathcal{I}_{\mathbf{j}}:=\{\mathbf{i} \in \mathbf{I}: \mathrm{int}(Q_{\mathbf{j}} \cap Q'_{\mathbf{i}})\neq \emptyset\}.
\end{align}
The  corresponding spline space  is defined as
\begin{equation*}
\widehat{\mathcal{S}}_{h} =\widehat{\mathcal{S}}_{h}([0,1]^d)  := \mathrm{span}\left\{B_{\mathbf{i}} : \, \mathbf{i} \in \mathbf{I} \right\},
\end{equation*} 
where $h$ is the maximal mesh-size in all knot vectors, that is
\begin{equation*}
	h:= \max_{\substack{1 \leq k \leq d \\ 1 \leq i \leq m+p+1 }}\{ | \xi_{k,i+1} - \xi_{k,i} |\}.
\end{equation*}
\begin{ass}\label{ass:quasi_uniform_mesh}
	We assume that the knot vectors are quasi-uniform, that is, there exists  $\alpha > 0$, independent
	of $h$, such that each nonempty knot span $( \xi_{k,i} ,
	\xi_{k,i+1})$ fulfils $ \alpha h \leq \xi_{k,i+1} - \xi_{k,i}$, for $1\leq k \leq d$.
\end{ass}
A family of linear functionals $\{  \widehat \varphi_{\mathbf{i}} \}_{\mathbf{i} \in \mathbf{I}}$ is called a \emph{dual basis} for the set of tensor-product B-splines $\widehat{\mathcal{S}}_{h}$ if it verifies
\begin{align*}
	 \widehat \varphi_{\mathbf{i}} (\widehat B_{\mathbf{j}}) = \delta_{\mathbf{i} \mathbf{j}},
\end{align*}
where $\delta_{\mathbf{i} \mathbf{j}}$ is the Kronecker delta.
\begin{thm}{\cite[Theorem 12.5]{schumaker_2007}} \label{thm_dualbasis_param}
	There exists a dual basis and a positive constant $C$, independent of $h$, satisfying
	\begin{align*}
		| \widehat \varphi_{\mathbf{i}}(\widehat u) | \leq C h^{- \frac{d}{2}} \| \widehat u \|_{L^2({Q'_{\mathbf{i}}})}, & & \forall \, \widehat u \in L^2((0,1)^d) \text{ and } \forall \, \mathbf{i} \in \mathbf{I}.
	\end{align*} 
\end{thm}
\subsection{Isogeometric space on a patch}
Now, we consider a single patch domain $\Omega \subset \mathbb{R}^d$, given by a $d$-dimensional spline parametrization $\vett[F]$, that is
\begin{align*}
	\Omega = \boldsymbol{F} (\widehat{\Omega}), & & \text{with } \boldsymbol{F}(\boldsymbol{\xi}) = \sum_{\mathbf{i}}  \boldsymbol{C}_{\mathbf{i}}  \widehat  B_{\mathbf{i}} (\boldsymbol{\xi}),
\end{align*}
where $\boldsymbol{C}_{\mathbf{i}} $ are the control points and
$   \widehat B_{\mathbf{i}}$ are tensor-product B-spline basis
functions defined on the parametric patch $\widehat{\Omega}:=(0,1)^d$.
{\Ru In the setting of this paper, $\Omega $ and its parametrization $\vett[F]$ do  not change when $h$- and $p$-refinements are performed}.
\begin{ass}\label{ass:reg_singlepatch}
	Let $\vett[F] \in C^1([0,1]^d)$ and assume that for all $\vett[x] \in [0,1]^d$ \linebreak$\jac{\vett[F](\vett[x])} >0$.
\end{ass}
Following the 
isoparametric paradigm, the isogeometric basis functions $B_{\mathbf{i}}$
are defined as $B_{\mathbf{i}} =  \widehat B_{\mathbf{i}}\circ
\boldsymbol{F} ^{-1}$. 
Thus, the isogeometric space on $\Omega$ is defined as
\begin{equation*}\label{eq:disc_space}
\mathcal{S}_{h} = \mathcal{S}_{h}(\Omega) := \mathrm{span}\left\{  B_{\mathbf{i}}:=\widehat B_{\mathbf{i}} \circ \mathbf{F}^{-1} \ : \ \mathbf{i} \in \mathbf{I}   \right\}. 
\end{equation*}
\begin{cor}\label{cor_dualbasis_physical}
	The family $\{ \varphi_{\mathbf{i}} \}_{\mathbf{i} \in \mathbf{I}}$, defined as
	\begin{align*}
		 \varphi_{\mathbf{i}}(u) := \widehat \varphi_{\mathbf{i}}(u \circ \mathbf{F}), & & \forall \, \mathbf{i} \in \mathbf{I},
	\end{align*}
	is a dual basis for $\mathcal{S}_{h}$. Moreover, there exists a positive constant $C$, independent of $h$, satisfying
	\begin{align*}
	|  \varphi_{\mathbf{i}}( u) | \leq C h^{- \frac{d}{2}} \|  u \|_{L^2(\mathbf{F}(Q'_{\mathbf{i}}))}, & & \forall \,  u \in L^2(\Omega) \text{ and } \forall \, \mathbf{i} \in \mathbf{I}.
	\end{align*}
\end{cor}
\begin{proof}
	The inequality is obtained from Theorem \ref{thm_dualbasis_param} by a standard change of variables.
\end{proof}

\begin{prop}\label{prop:stability}
	There exists a positive constant $C$, independent of $h$ and $p$, such that
	\begin{align*}
		\| u \|_{L^2{(\vett[F](Q_{\mathbf{i}}))}} \leq C h^{\frac{d}{2}} \max_{\mathbf{j} \in \mathcal{I}_{\mathbf{i}}} |\varphi_{\mathbf{j}}(u)|, & & \forall \,  u \in L^2(\Omega) \text{ and } \forall \, \mathbf{i} \in \mathbf{I}_e.
	\end{align*}
\end{prop}
\begin{proof}
	Using the extension to $d$-variate isogeometric functions of the partition of unity property \eqref{eq:partition_unity}, it holds
	\begin{align*}
		\| u \|^2_{L^2(\vett[F](Q_{\mathbf{i}}))} & =  \int_{ \vett[F](Q_{\mathbf{i}})} \left( \sum_{\mathbf{i} \in \mathcal{I}_{\mathbf{i}}} \varphi_{\mathbf{i}}(u) B_{\mathbf{i}}(\vett[x]) \right)^2 d \vett[x] \\
		&\leq \int_{ \vett[F](Q_{\mathbf{i}})} \left( \max_{\mathbf{j} \in \mathcal{I}_{\mathbf{i}}} |\varphi_{\mathbf{j}}(u)|  \sum_{\mathbf{j} \in \mathcal{I}_{\mathbf{i}}} B_{\mathbf{i}}(\vett[x]) \right)^2 d \vett[x] \\
		& = |\vett[F](Q_{\mathbf{i}})|	\max_{\mathbf{j} \in \mathcal{I}_{\mathbf{i}}} | \varphi_{\mathbf{j}}(u)|^2 \leq C h^{d} \max_{\mathbf{j} \in \mathcal{I}_{\mathbf{i}}} |\varphi_{\mathbf{j}}(u)|^2.
	\end{align*}
\end{proof}
By introducing a co-lexicographical reordering of the basis functions,
with a minor abuse of notation we will also write in what follows 
\begin{equation}\label{def:parametric_spline}
\mathcal{S}_{h}= \spann \left\{  B_{\mathbf{i}} : \, \mathbf{i} \in \mathbf{I} \right\} = \spann\left\{ B_{i} \right\}_{i=1}^{\Ndof}.
\end{equation}
\subsection{Isogeometric spaces on a multipatch domain}\label{sec:mp}
We follow the notation of \cite{daveiga_buffa_sangalli_2014}.
A multipatch domain $\Omega \subset \mathbb{R}^d$ is an open set, defined as the union of $\NOmega$ subdomains,
\begin{equation}\label{def:patches}
	\overline{\Omega}= \bigcup_{{\ptca}=1}^{\NOmega} \overline{\Omega^{({\ptca})}},
\end{equation} 
where the subdomains $\Omega^{({\ptca})}=\vett[F]^{({\ptca})}(\widehat{\Omega})$ are referred to as patches and are assumed to be disjoint. Each $\vett[F]^{({\ptca})}$ is a different spline parametrization that satisfies the following assumption.
\begin{ass}\label{ass:multipatch}
	Let $\vett[F]^{(\ptca)} \in C^1([0,1]^d)$ and assume that for all $\vett[x] \in [0,1]^d$, $\jac{\vett[F]^{(\ptca)}(\vett[x])} >0$, for all $\ptca=1, \ldots, \NOmega$.
\end{ass}
In the following, the superindex $({\ptca})$ will identify entities that refer to $\Omega^{({\ptca})}$. 
Then, following the same construction as above, we introduce, for each patch $\Omega^{({\ptca})}$,  B-spline spaces
\begin{align*}
	\widehat{\mathcal{S}}^{({\ptca})}_{h}:= \spann\left\{ \widehat{B}^{({\ptca})}_{i} : \, i=1, \ldots , \Ndof^{({\ptca})} \right\}. 
\end{align*}
and isogeometric spaces
\begin{align*}
	\mathcal{S}^{({\ptca})}_{h}:= \spann\left\{  B^{({\ptca})}_{i} \ : \ i=1,\ldots , N_{\text{dof}}^{({\ptca})}   \right\}.
\end{align*}
We assume for simplicity that the degree $p$ is the same for all patches.
For the definition of the isogeometric space in the whole $\Omega$, we further impose continuity at the interfaces between patches, that is
\begin{equation}\label{def:global_space}
	V_h:=\left\{ v \in C^0(\Omega) : v|_{\Omega^{({\ptca})}} \in \mathcal{S}^{({\ptca})}_{h} \text{ for }{\ptca}=1, \ldots , \NOmega \right\}.
\end{equation}
To construct a basis for space $V_{h}$, we introduce a suitable conformity assumption. For all $\ptca, \ptcb \in \{ 1,\ldots , \NOmega \}$, with $\ptca \neq \ptcb$, let $\Gamma_{{\ptca} {\ptcb}}=\partial \Omega^{({\ptca})} \cap \partial \Omega^{({\ptcb})}$ be the interface between the patches $\Omega^{({\ptca})}$ and $\Omega^{({\ptcb})}$.
\begin{ass}\label{ass:conformity}
	We assume:
	\begin{enumerate}
		\item $\Gamma_{{\ptca}{\ptcb}}$ is either a vertex or the image of a full edge or the image of a full face for both parametric domains.
		\item For each $B^{({\ptca})}_{\mathbf{i}} \in \mathcal{S}_{h}^{({\ptca})}$ such that $\supp(B^{({\ptca})}_{\mathbf{i}}) \cap \Gamma_{{\ptca}{\ptcb}} \neq \emptyset$, there exists a function $B^{({\ptcb})}_{\mathbf{j}} \in \mathcal{S}_{h}^{({\ptcb})}$ such that $B^{({\ptca})}_{\mathbf{i}} |_{ \Gamma_{{\ptca}{\ptcb}}}=B^{({\ptcb})}_{\mathbf{j}} |_{\Gamma_{{\ptca}{\ptcb}}}$.
	\end{enumerate}
\end{ass}
We define, for each patch $\Omega^{({\ptca})}$, an application
\begin{equation*}
	G^{({\ptca})}: \{ 1, \ldots , \Ndof^{(\ptca)} \} \rightarrow \mathcal{J}=\{ 1, \ldots , \dim(V_{h}) \},
\end{equation*}
in such a way that $G^{({\ptca})}(i)=G^{({\ptcb})}(j)$ if and only if $\Gamma_{{\ptca}{\ptcb}} \neq \emptyset$ and \linebreak \mbox{$B^{({\ptca})}_{i} |_{\Gamma_{{\ptca}{\ptcb}}}=B^{({\ptcb})}_{j} |_{\Gamma_{{\ptca}{\ptcb}}}$}. Moreover, we define, for each global index $l \in \mathcal{J}$, the set of pairs $\mathcal{J}_l:= \{ ({\ptca},i): \, G^{({\ptca})}(i)=l \}$, which collects the local indices of patchwise contributions to the global function, and the scalar 
\begin{equation}\label{def:n_l}
	n_l:= \# \mathcal{J}_l,
\end{equation}
that expresses the patch multiplicity for the global index $l$.
Furthermore, let 
\begin{equation}\label{def:N_patch}
\Npatch:=\max \{ n_l : \, l \in \mathcal{J} \}
\end{equation}
be the maximum number of adjacent patches (i.e., whose closure has non-empty intersection).
We define, for each $l \in \mathcal{J}$, the global basis function
\begin{equation}\label{def:mp_basis}
	B_l(\vett[x]):= \begin{cases}
						B^{({\ptca})}_{i}(\vett[x]) & \text{ if } \vett[x] \in \overline{\Omega ^{({\ptca})}} \text{ and } ({\ptca},i) \in \mathcal{J}_l, \\
						0 & \text{ otherwise},
					\end{cases}
\end{equation} 
which is continuous due to Assumption \ref{ass:conformity}. Then
\begin{equation}\label{def:V_h}
	V_h= \spann \{ B_l: \, l \in \mathcal{J} \}.
\end{equation}
The set $\{ B_l: \, l \in \mathcal{J} \}$
where $B_l$ is defined as in \eqref{def:mp_basis}, represents a basis for $V_{h}$.
Finally, we also introduce the index set $\mathcal{J}^{(\ptca)} \subset \mathcal{J}$ such that $l \in \mathcal{J}^{(\ptca)}$ if and only if $l=G^{(\ptca)}(i)$ for some $i$. Clearly $\# \mathcal{J}^{(\ptca)}= \Ndof^{(\ptca)}$ and $\mathcal{J}^{(\ptca)}$ can be used directly as index set for $\widehat{\mathcal{S}}_h^{(\ptca)}$ and $\mathcal{S}_h^{(\ptca)}$, with minor abuse of notation.

\subsection{Kronecker product}
The Kronecker product of two matrices $\mathbf{A}\in\mathbb{C}^{n_1\times n_2}$ and $\mathbf{B}\in\mathbb{C}^{n_3\times n_4}$ is defined as
\begin{equation*}
\mathbf{A} \otimes \mathbf{B}:=\begin{bmatrix}
[\mathbf{A}]_{1,1}\mathbf{B}  & \dots& [\mathbf{A}]_{1,n_2}\mathbf{B}\\
\vdots& \ddots &\vdots\\
[\mathbf{A}]_{n_1, 1}\mathbf{B}& \dots & [\mathbf{A}]_{n_1, n_2}\mathbf{B}
\end{bmatrix}\in \mathbb{C}^{n_1n_3\times n_2 n_4},
\end{equation*}
where the $ij$-th entry of the matrix $\mathbf{A}$ is denoted by $[\mathbf{A}]_{i,j}$.
The most important properties of the Kronecker product that we will exploit in this work are the following:
\begin{itemize}
	\item if $\mathbf{A}$, $\mathbf{B}$, $\mathbf{C}$ and $\mathbf{D}$ are matrices of conforming order, then  it holds
	\begin{equation}
	\label{eq:kron_prod}
	(\mathbf{A}\otimes \mathbf{B}) \cdot (\mathbf{C}\otimes \mathbf{D}) = (\mathbf{AC}) \otimes (\mathbf{BD});
	\end{equation}
	\item if $\mathbf{A}$ and $\mathbf{B}$ are non-singular, then
	\begin{equation}
	\label{eq:kron_inv}
	(\mathbf{A}\otimes \mathbf{B})^{-1}=\mathbf{A}^{-1}\otimes\mathbf{B}^{-1}.
	\end{equation}  
\end{itemize}
Finally, we recall that the matrix-vector product  can be efficiently computed
for a  matrix that has a Kronecker product structure. For this
purpose we define, for $m=1,\dots,d,$ the $m$-mode product  $\times_m$  of a tensor $\mathbf{X}\in\mathbb{C}^{n_1\times\dots\times n_{d}}$ with a matrix $\mathbf{M}\in\mathbb{C}^{k \times n_m}$ as a tensor of size $n_1\times\dots\times n_{m-1}\times k \times n_{m+1}\times \ldots n_{d}$ whose elements are  
\begin{equation*}
\left[ \mathbf{X}\times_m \mathbf{M} \right]_{i_1, \dots, i_{d}} = \sum_{j=1}^{n_m} [\mathbf{X}]_{i_1,,\dots, i_{m-1},j,i_{m+1}\dots,i_{d}}[\mathbf{M}]_{i_m,j }.
\end{equation*}
Then, given $\mathbf{M}_i\in\mathbb{C}^{k_i\times n_i}$ for $i=1,\dots, d$, it holds
\begin{equation}\label{eq:kronecker_vec_prod}
\left(\mathbf{M}_{d}\otimes\dots\otimes \mathbf{M}_1\right)\mathrm{vec}\left(\mathbf{X}\right)=\mathrm{vec}\left(\mathbf{X}\times_1 \mathbf{M}_1\times_2 \dots \times_{d}\mathbf{M}_{d} \right),
\end{equation} 
where the vectorization operator ``vec''   applied to a tensor stacks its entries  into a column vector as
\begin{align*}
	[\mathrm{vec}(\mathbf{X})]_{j}=[\mathbf{X}]_{i_1,\ldots,i_{d}},
\end{align*}
for $i_l=1,\dots,n_{l}$, $l=1,\ldots,d$  and
\begin{equation*}
j=i_1+\sum_{k=2}^{d}\left[(i_k-1)\Pi_{l=1}^{k-1}n_l\right].
\end{equation*}
For more details on Kronecker product we refer to \cite{Kolda2009}.

\section{Mass preconditioner on a patch} \label{sec:single_patch}
In this section we propose a preconditioner for the Galerkin mass matrix associated to a single patch domain, denoted $\Omega$, that is
{\Rd \begin{equation}\label{eq:mass_matrix}
	[\M]_{i,j}=\int_{\widehat{\Omega}}  \widehat{B}_{i} (\vett[x]) \widehat{B}_{j}(\vett[x]) |\jac{\vett[F](\vett[x])}|d \vett[x].
\end{equation}}
Generalizing, we will consider
{\Rd \begin{align}\label{eq:mass_weighted}
[\M]_{i,j}=\int_{\widehat{\Omega}}  \widehat{B}_{i} (\vett[x]) \widehat{B}_{j} (\vett[x]) \omega (\vett[x])d \vett[x]
\end{align}}
for a weight $\omega$ that fulfils the following assumption.
\begin{ass}\label{ass:g_regularity}
	We assume $\omega \in C^0([0,1]^d)$	and $\omega(\vett[x]) > 0$, for all ${\vett[x] \in [0,1]^d}$.
\end{ass}
Let 
\begin{equation*}
	\omega_{\mathrm{min}}=\min_{\vett[x] \in [0,1]^d} \omega(\vett[x]),
\end{equation*}
that, thanks to Assumption \ref{ass:g_regularity}, is strictly positive.
Furthermore, thanks to Heine-Cantor theorem, the function $\omega$ is uniformly continuous, that is there exists a non-decreasing $\mu:[0,\infty) \rightarrow [0,\infty)$ such that
\begin{align} \label{eq:uniform_cont_1}
	|\omega(\vett[x]_1)-\omega(\vett[x]_2)| \leq \mu(|\vett[x]_1-\vett[x]_2|), & & \forall \vett[x]_1, \vett[x]_2 \in \widehat{\Omega}
\end{align}
and
\begin{align} \label{eq:uniform_cont_2}
 \lim_{t \rightarrow 0^+}\mu(t)=0.
\end{align}
As a preconditioner for the mass matrix $\M$, defined in \eqref{eq:mass_weighted}, we consider
\begin{equation}\label{def:single_patch_prec}
\P:=\Prs,
\end{equation}
where
\begin{align}\label{def:matrices_1}
[\widehat{\M}]_{i,j}:=\int_{\widehat{\Omega}} \widehat{B}_{i}(\vett[x]) \widehat{B}_{j}(\vett[x]) d \vett[x], & & \widehat{\D}:=\text{diag}\left( \widehat{\M} \right),  & & \D:=\text{diag}\left( \M \right).
\end{align}
From now on, 
given $u \in \mathcal{S}_h$, we will denote by $\vett[u]$ the vector
containing the coordinates of $u$ with respect to spline basis.
{\Ru \begin{lem}\label{lem:cond_bound_param}
	There exist a constant $\widehat{C} >0$, independent of $h$, such that
	\begin{equation}\label{eq:bound-on-widehatM}
		\widehat{C} h^d \leq
                \lambda_{\mathrm{min}}(\widehat{\M}) \leq
                \lambda_{\mathrm{max}}(\widehat{\M}) \leq   h^d,
	\end{equation}
	where $\lambda_{\mathrm{max}}(\widehat{\M})$ and $\lambda_{\mathrm{min}}(\widehat{\M})$ are the maximum and minimum eigenvalue of $\widehat{\M}$.
\end{lem}
\begin{proof}
Recalling Assumption \ref{ass:quasi_uniform_mesh}, from the classical
result \cite[Theorem 5.1-5.2]{Boor1976Splines}, for $d=1$ we get
\begin{displaymath}
  \frac{\alpha h}{4(p+1)^3 9^{p}} \sum_{i=1}^{m} v^2_i  \leq \left \| 	\sum_{i=1}^{m} v_i \widehat b_{i,p} (x) \right \|
  _{L^2(0,1)}^2 \leq h \sum_{i=1}^{m} v^2_i 
\end{displaymath}
where $\alpha$ is the quasi-uniformity constant from Assumption
\ref{ass:quasi_uniform_mesh}. The bounds \eqref{eq:bound-on-widehatM} follow  by tensorization and
applying the Courant-Fischer theorem.  
\end{proof}}

\begin{cor}\label{cor:cond_bound_phys}
	Under Assumption \ref{ass:g_regularity}, there exist two positive {\Rd constants} $C_1,C_2$, independent of $h$, such that
	\begin{equation}\label{eq:est_eigenvalue_phys}
	C_1h^d \leq \lambda_{\mathrm{min}}({\M}) \leq \lambda_{\mathrm{max}}({\M}) \leq C_2 h^d.
	\end{equation}
\end{cor}
\begin{proof}
	We observe that
	\begin{align*}
		\lambda_{\mathrm{min}}({\M}) & = \min_{\vett[v] \neq \vett[0]} \frac{\vett[v]^T \M \vett[v]}{\vett[v]^T \vett[v]} = \min_{\vett[v] \neq \vett[0]} \frac{\int_{\widehat{\Omega}}\left( \sum_{i=1}^{\Ndof} \vett[v]_i \widehat{B}_i(\vett[x]) \right)^2 \omega(\vett[x]) d \vett[x]}{\vett[v]^T \vett[v]} \\
		& \geq \min_{\vett[v] \neq \vett[0]} \frac{\int_{\widehat{\Omega}}\left( \sum_{i=1}^{\Ndof} \vett[v]_i \widehat{B}_i(\vett[x]) \right)^2 d \vett[x]}{\vett[v]^T \vett[v]} \inf_{\vett[x] \in \widehat{\Omega}} \omega(\vett[x]) = \lambda_{\mathrm{min}}({\widehat{\M}}) \inf_{\vett[x] \in \widehat{\Omega}} \omega(\vett[x])
	\end{align*}
	and similarly
	\begin{align*}
		\lambda_{\mathrm{max}}({\M}) & = \max_{\vett[v] \neq \vett[0]} \frac{\vett[v]^T \M \vett[v]}{\vett[v]^T \vett[v]} \leq \lambda_{\mathrm{max}}({\widehat{\M}}) \sup_{\vett[x] \in \widehat{\Omega}} \omega(\vett[x]).
	\end{align*}
	Thanks to Assumption \ref{ass:g_regularity}, $\omega$ is
        bounded from below and above by two positive {\Rd constants}
        $\omega_{\mathrm{min}}$ and $\omega_{\mathrm{max}}$, thus
        exploiting Lemma \ref{lem:cond_bound_param}, we obtain the
        thesis with \Ru ${C}_1=\omega_{\mathrm{min}} \widehat{C}$ and ${C}_2=\omega_{\mathrm{max}}$.
\end{proof}

\begin{rem}
	Thanks to Courant-Fischer theorem, the last inequality in \eqref{eq:est_eigenvalue_phys} can be rewritten as follows
	\begin{align}\label{eq:est_2_bis}
	\|u\|^2_{L^2(\Omega)} \leq C_2 h^d \sum_{\mathbf{i} \in \mathbf{I}} |\varphi_{\mathbf{i}}(u)|^2, & & \forall u \in \mathcal{S}_{h}, 
	\end{align}
	where $\{ \varphi_{\mathbf{i}} \}_{\mathbf{i} \in \mathbf{I}}$ denotes the dual basis introduced in Corollary \ref{cor_dualbasis_physical}.
\end{rem}

\begin{cor}\label{cor:cond_bound_prec}
	Under Assumption \ref{ass:g_regularity}, there exist two positive {\Rd constants} $\widetilde{C}_1,\widetilde{C}_2$, independent of $h$, such that
	\begin{equation*}
	\widetilde{C}_1h^d \leq \lambda_{\mathrm{min}}({\P}) \leq \lambda_{\mathrm{max}}({\P}) \leq \widetilde{C}_2 h^d.
	\end{equation*}
\end{cor}

\begin{proof} {\Rd
It holds
\begin{equation}\label{eq:lmin_e}
\lambda_{\mathrm{min}}(\P) = \lambda_{\mathrm{min}}(\Prs) \geq \frac{\lambda_{\mathrm{min}}(\D) \lambda_{\mathrm{min}}(\widehat{\M})}{\lambda_{\mathrm{max}}(\widehat{\D})}
\end{equation}
Since $\D$ and $\widehat{\D}$ are diagonal, their eigenvalues correspond to their diagonal entries. We have
\begin{align*}
[\D]_{i,i}=[\M]_{i,i}= \vett[e]_i^T \M \vett[e]_i & & \text{and} & & [\widehat{\D}]_{i,i}=[\widehat{\M}]_{i,i}= \vett[e]_i^T \widehat{\M} \vett[e]_i,
\end{align*}
 where $\vett[e]_i$ denotes the $i$-th vector of the standard basis. Thus, it holds
\begin{align}\label{eq:eigenvalues_diagonals}
\lambda_{\mathrm{min}}(\M) \leq [\D]_{i,i} \leq \lambda_{\mathrm{max}}(\M) & & \text{and} & & \lambda_{\mathrm{min}}(\widehat{\M}) \leq [\widehat{\D}]_{i,i} \leq \lambda_{\mathrm{max}}(\widehat{\M}).
\end{align}
In this way, equation \eqref{eq:lmin_e} becomes
\begin{align*}
\lambda_{\text{min}}(\P) \geq \frac{\lambda_{\mathrm{min}}(\M) \lambda_{\mathrm{min}}(\widehat{\M})}{\lambda_{\mathrm{max}}(\widehat{\M})},
\end{align*}
Combining the latter inequality with Lemma \ref{lem:cond_bound_param} and Corollary \ref{cor:cond_bound_phys}, it follows that there exists a constant $\widetilde{C}_1$, independent of $h$, such that
\begin{equation*}
\lambda_{\text{min}}(\P) \geq \widetilde{C}_1h^d.
\end{equation*}
The upper bound on $\lambda_{\mathrm{max}}(\M)$ is derived in a similar way.
} \end{proof}

{\Rd
The condition number of a symmetric positive definite matrix $\mathbf{A}$ is defined as
\begin{align}\label{def:cond}
	\kappa(\mathbf{A}):= \frac{\lambda_{\mathrm{max}}(\mathbf{A})}{\lambda_{\mathrm{min}}(\mathbf{A})},
\end{align}
We observe that, under Assumption \ref{ass:g_regularity}, there exists a constant $C$, independent of $h$, such that
	\begin{equation}\label{eq:thesis2}
	\kappa(\PMP)  \leq C.
	\end{equation}
	Indeed, recalling Courant-Fischer theorem, it follows
	\begin{align*}
	\begin{split}
	\lambda_{\mathrm{min}}(\PMP) = \min_{\vett[v] \neq \vett[0]} \frac{\vett[v]^T \M \vett[v]}{\vett[v]^T \P \vett[v]} \geq \frac{\lambda_{\mathrm{min}}(\M)}{\lambda_{\mathrm{max}}(\P)}
	\end{split}
	\end{align*}
	and similarly
	\begin{equation*}
	\lambda_{\mathrm{max}}(\PMP)=\max_{\vett[v] \neq \vett[0]} \frac{\vett[v]^T \M \vett[v]}{\vett[v]^T \P \vett[v]} \leq \frac{\lambda_{\mathrm{max}}(\M)}{\lambda_{\mathrm{min}}(\P)}.
	\end{equation*}
	Then \eqref{eq:thesis2} follows from Corollary \eqref{cor:cond_bound_phys} and Corollary \eqref{cor:cond_bound_prec}.
This estimate can be improved when $h$ approaches zero, as stated in the next result.}

\begin{thm} \label{thm:single_patch_h}
Under Assumption \ref{ass:g_regularity}, it holds
\begin{equation}
\label{eq:limit}
	\lim_{h \rightarrow 0}\kappa(\PMP) =1.
\end{equation}
\end{thm}
\begin{proof}
	Recalling Courant-Fischer theorem, it follows
	\begin{align}\label{eq:lmin}
	\begin{split}
	\lambda_{\mathrm{min}}(\PMP) & = \min_{\vett[v] \neq \vett[0]} \frac{\vett[v]^T \M \vett[v]}{\vett[v]^T \P \vett[v]} \\
	& =\min_{\vett[w] \neq \vett[0]} \frac{\vett[w]^T \DMD \vett[w]}{\vett[w]^T \Mp \vett[w]}=\min_{\vett[w] \neq \vett[0]} \frac{\vett[w]^T \Mt \vett[w]}{\vett[w]^T \Mp \vett[w]}
	\end{split}
	\end{align}
	and similarly
	\begin{equation}\label{eq:lmax}
	\lambda_{\mathrm{max}}(\PMP)=\max_{\vett[w] \neq \vett[0]} \frac{\vett[w]^T \Mt \vett[w]}{\vett[w]^T \Mp \vett[w]},
	\end{equation}
	where we have defined $\Mt:=\DMD$. \\
	The entries of the matrix $\Mt$ can be rewritten as
	\begin{align*}
	[\Mt]_{i,j}&=[ \DMD ]_{i,j}\\
	&=\frac{\| \widehat{B_i} \|_{L^2(\widehat{\Omega})} \| \widehat{B_j} \|_{L^2(\widehat{\Omega})}}{\| \sqrt{\omega} \widehat{B_i} \|_{L^2(\widehat{\Omega})} \| \sqrt{\omega} \widehat{B_j} \|_{L^2(\widehat{\Omega})}}\int_{\widehat{\Omega}} \omega(\vett[x]) \widehat{B_i}(\vett[x]) \widehat{B_j}(\vett[x]) d \vett[x].
	\end{align*}
	We observe that for all $i \in \{ 1 \ldots , \Ndof \}$, it holds
	\begin{equation*}
		[\Mt]_{i,i}=\frac{\| \widehat{B_i} \|_{L^2(\widehat{\Omega})}^2}{\| \sqrt{\omega} \widehat{B_i} \|_{L^2(\widehat{\Omega})}^2 }\int_{\widehat{\Omega}} \omega(\vett[x]) \widehat{B_i}^2(\vett[x]) d \vett[x] = \| \widehat{B_i} \|_{L^2(\widehat{\Omega})}^2 =[\Mp]_{i,i}.  
	\end{equation*}
	Let us now consider $i,j \in \{ 1 \ldots , \Ndof \}$, with $i \neq j$.
	From equation \eqref{eq:uniform_cont_1}, it follows
	\begin{equation*} \label{eq:jacobian_lipschitz}
	\omega(\vett[x]_1) \leq \omega(\vett[x]_2)+\mu(|\vett[x]_1-\vett[x]_2|), \quad \forall \,\vett[x]_1,\vett[x]_2 \in \widehat{\Omega}.
	\end{equation*}
	As a consequence, by observing that $\text{diam}\left( \supp(  \widehat{B_i}) \cap \supp(\widehat{B_j}) \right) \leq h(p+1) \sqrt{d}$, for all $i,j=1,\ldots, \Ndof$ with $i \neq j$, and denoting 
	\begin{equation*}
		\vett[x]_{ij}:= \text{argmin}\left\lbrace \omega(\vett[x]) : \vett[x] \in \supp(  \widehat{B_i}) \cap \supp(  \widehat{B_j}) \right\rbrace,
	\end{equation*}
	we obtain 
	\begin{align}\label{eq:estimate_1}
	\begin{split}
	\omega (\vett[x]_{ij}) \int_{\widehat{\Omega}} \widehat{B}_i(\vett[x]) \widehat{B}_j(\vett[x]) d \vett[x] &\leq \int_{ \widehat{\Omega}} \omega(\vett[x]) \widehat{B}_i(\vett[x]) \widehat{B}_j(\vett[x]) d \vett[x] \\
	& \leq \left( \omega (\vett[x]_{ij}) + \mu(h(p+1)\sqrt{d}) \right) \int_{\widehat{\Omega}} \widehat{B}_i(\vett[x]) \widehat{B}_j(\vett[x]) d \vett[x].
	\end{split}
	\end{align}
	Similarly, for all $i,j=1,\ldots,\Ndof$ with $i \neq j$ and such that \linebreak
	${\supp(  \widehat{B_i}) \cap \supp(  \widehat{B_j}) \neq \emptyset}$, it holds
	\begin{align}\label{eq:estimate_2}
	\begin{split}
\omega(\vett[x]_{ij}) \| \widehat{B_i} \|_{L^2(\widehat{\Omega})} \| \widehat{B_j} \|_{L^2(\widehat{\Omega})} & \leq \| \sqrt{\omega} \widehat{B_i} \|_{L^2(\widehat{\Omega})} \| \sqrt{\omega} \widehat{B_j} \|_{L^2(\widehat{\Omega})}\\ 
& \leq \left( \omega(\vett[x]_{ij}) + \mu(h(p+1)\sqrt{d}) \right) \| \widehat{B_i} \|_{L^2(\widehat{\Omega})} \| \widehat{B_j} \|_{L^2(\widehat{\Omega})}.
	\end{split}
	\end{align}
	Thus, by combining inequalities \eqref{eq:estimate_1} and \eqref{eq:estimate_2}, we have the following bounds for the entries of $\Mt$:
	\begin{align}\label{eq:est_entries}
	\begin{split}
	 [\Mt]_{i,j} =& \frac{\| \widehat{B_i} \|_{L^2(\widehat{\Omega})} \| \widehat{B_j} \|_{L^2(\widehat{\Omega})}}{\| \sqrt{\omega} \widehat{B_i} \|_{L^2(\widehat{\Omega})} \| \sqrt{\omega} \widehat{B_j} \|_{L^2(\widehat{\Omega})}}\int_{\widehat{\Omega}} \omega(\vett[x]) \widehat{B_i}(\vett[x]) \widehat{B_j}(\vett[x]) d \vett[x] \\ \leq& \frac{\omega(\vett[x]_{ij}) + \mu(h(p+1)\sqrt{d}) }{ \omega (\vett[x]_{ij}) } \int_{\widehat{\Omega}} \widehat{B}_i(\vett[x]) \widehat{B}_j(\vett[x]) d \vett[x], \\
	[\Mt]_{i,j} =&  \frac{\| \widehat{B_i} \|_{L^2(\widehat{\Omega})} \| \widehat{B_j} \|_{L^2(\widehat{\Omega})}}{\| \sqrt{\omega} \widehat{B_i} \|_{L^2(\widehat{\Omega})} \| \sqrt{\omega} \widehat{B_j} \|_{L^2(\widehat{\Omega})}}\int_{\widehat{\Omega}} \omega(\vett[x]) \widehat{B_i}(\vett[x]) \widehat{B_j}(\vett[x])d \vett[x] \\ \geq& \frac{ \omega (\vett[x]_{ij})}{ \omega ({\Ru \vett[x]_{ij}}) + \mu(h(p+1)\sqrt{d})} \int_{\widehat{\Omega}} \widehat{B}_i(\vett[x]) \widehat{B}_j(\vett[x]) d \vett[x].
	\end{split}
	\end{align}
	Having defined 
	\begin{equation*}
	\sigma:= \frac{\mu(h(p+1)\sqrt{d})}{ \omega_{\mathrm{min}}},
	\end{equation*}
	we observe that
	\begin{equation}\label{eq:sigma}
		\lim_{h \rightarrow 0} \sigma =0
	\end{equation}
	and
	\begin{align}\label{eq:est_omega}
	\begin{split}
	\frac{ \omega (\vett[x]_{ij}) + \mu(h(p+1)\sqrt{d})}{ \omega(\vett[x]_{ij}) } & =1+\frac{\mu(h(p+1)\sqrt{d})}{\omega(\vett[x]_{ij})} \leq 1+\sigma, \\
	\frac{ \omega (\vett[x]_{ij}) }{ \omega (\vett[x]_{ij}) + \mu(h(p+1)\sqrt{d})} & =1-\frac{\mu(h(p+1)\sqrt{d})}{\omega(\vett[x]_{ij}) + \mu(h(p+1)\sqrt{d})} \geq 1-\sigma.
	\end{split}
	\end{align}
	Collecting  \eqref{eq:mass_matrix}, \eqref{eq:est_entries} and \eqref{eq:est_omega}, we can bound the entry-wise distance between the matrices $\Mt$ and $\Mp$ as follows
	\begin{equation}\label{eq:entry_pert}
	 - \sigma [\Mp]_{i,j}\leq [\Mt]_{i,j} - [\Mp]_{i,j} \leq \sigma [\Mp]_{i,j}.
	\end{equation}
	For all $\vett[w] \neq \vett[0]$, it holds
	\begin{equation}\label{eq:rayleigh_1}
		\frac{\vett[w]^T \Mt \vett[w]}{\vett[w]^T \Mp \vett[w]}= \frac{\vett[w]^T (\Mt + \Mp - \Mp) \vett[w]}{\vett[w]^T \Mp \vett[w]} = 1 + \frac{\vett[w]^T (\Mt - \Mp) \vett[w]}{\vett[w]^T \Mp \vett[w]}.
	\end{equation}
	Exploiting equation \eqref{eq:entry_pert}, it follows
\begin{align}\label{eq:rayleigh_2}
\begin{split}
\sup_{\vett[w] \neq \vett[0]} \frac{\vett[w]^T (\Mt - \Mp) \vett[w]}{\vett[w]^T \Mp \vett[w]} & = \sup_{\vett[w] \neq \vett[0]} \frac{\sum_{i,j=1}^N \vett[w]_i \vett[w]_j ([\Mt]_{i,j} - [\Mp]_{i,j})}{\vett[w]^T \Mp \vett[w]} \\
& \leq \sup_{\vett[w] \neq \vett[0]} \frac{\sum_{i,j=1}^N | \vett[w]_i \vett[w]_j | | [\Mt]_{i,j} - [\Mp]_{i,j}|}{\vett[w]^T \Mp \vett[w]} \\
& \leq \sup_{\vett[w] \neq \vett[0]} \frac{ \sigma \sum_{i,j=1}^N | \vett[w]_i \vett[w]_j |  [\Mp]_{i,j}}{\vett[w]^T \Mp \vett[w]} \\
& \leq \sigma \frac{\sup_{\vett[w] \neq \vett[0]} \frac{ | \vett[w] |^T  [\Mp] | \vett[w] |}{\vett[w]^T \vett[w]}}{\inf_{\vett[w] \neq \vett[0]} \frac{  \vett[w] ^T  [\Mp]  \vett[w] }{\vett[w]^T \vett[w]}} \leq \sigma \frac{\lambda_{\text{max}}(\Mp)}{\lambda_{\text{min}}(\Mp)}
\end{split}
\end{align}
where, in the last step, we have used the property that
\begin{align*}
	\sup \frac{ |\vett[w]|^T \Mp |\vett[w]|}{\vett[w]^T \vett[w]}= \sup \frac{ \vett[w]^T \Mp \vett[w]}{\vett[w]^T \vett[w]},
\end{align*}
where indeed the last sup is obtained for a $\vett[w]$ with non-negative entries, due to the fact that $\Mp$ has non-negative entries. 
Similarly
		\begin{align}\label{eq:rayleigh_3}
		\inf_{\vett[w] \neq \vett[0]} \frac{\vett[w]^T (\Mt - \Mp) \vett[w]}{\vett[w]^T \Mp \vett[w]} & \geq \inf_{\vett[w] \neq \vett[0]} \frac{ -\sigma \sum_{i,j=1}^N | \vett[w]_i \vett[w]_j |  [\Mp]_{i,j}}{\vett[w]^T \Mp \vett[w]} \\
		& \geq - \sigma \sup_{\vett[w] \neq \vett[0]} \frac{ \sum_{i,j=1}^N | \vett[w]_i \vett[w]_j |  [\Mp]_{i,j}}{\vett[w]^T \Mp \vett[w]} \geq  - \sigma \frac{\lambda_{\text{max}}(\Mp)}{\lambda_{\text{min}}(\Mp)}.
	\end{align} 
	From \eqref{eq:rayleigh_1}, \eqref{eq:rayleigh_2} and \eqref{eq:rayleigh_3} we get
	\begin{align}\label{eq:est_quot}
	1 - \sigma \kappa(\Mp) \leq \frac{\vett[w]^T \Mt \vett[w]}{\vett[w]^T \Mp \vett[w]} \leq 1 + \sigma \kappa (\Mp), & & \forall \vett[w] \neq \vett[0].
	\end{align}
	Finally, combining Lemma \ref{lem:cond_bound_param}, \eqref{def:cond}, \eqref{eq:uniform_cont_2}, \eqref{eq:lmin}, \eqref{eq:lmax}, \eqref{eq:sigma} and \eqref{eq:est_quot} we obtain
	\begin{equation*}\label{eq:estimate_3}
	\lim_{h \rightarrow 0} \kappa(\PMP) \leq \lim_{h \rightarrow 0} \frac{1+C \sigma}{1-C \sigma}=1.
	\end{equation*}
\end{proof}

\begin{rem} \label{rem:reg_F}
By applying Theorem \ref{thm:single_patch_h} to {\Ru the mass matrix \eqref{eq:mass_matrix}, if we assume that
$\vett[F] \in C^1([0,1]^d)$ and that $\omega (\vett[x]) =\jac{\vett[F](\vett[x])} > 0$} for all $\vett[x] \in [0,1]^d$, then the preconditioned matrix fulfils \eqref{eq:limit}.
\end{rem}

\begin{rem} \label{rem:cond_fun}
Theorem \ref{thm:single_patch_h} states that if we define
$$ \mu(h) :=  \kappa(\PMP) - 1,$$
then $\mu(h) = o(1)$ as $h \rightarrow 0$.
Clearly, the function $\mu$ may depend on other parameters beside the
mesh size, in particular it may depend on the spline degree $p$ and
the parametrization $\vett[F] $.
However, in all problems considered in Section \ref{sec:tests} our
numerical tests indicate that $\mu$ is linear with respect to $p$ and
mildly depends on $\vett[F] $, even when the geometry parametrization is singular.

\end{rem}
\B
\section{Mass preconditioner on multipatch domain}\label{sec:multipatch}
We now turn to examine a preconditioner for the mass matrix arising from multipatch domains, that is
{\Rd \begin{align}\label{eq:multipatch_mass}
	[\M]_{i,j}=\int_{ \Omega} B_i(\vett[x]) B_j(\vett[x]) d \vett[x], & & i,j \in \mathcal{J},
\end{align}}
where $\Omega$ is formed by the union of patches $\Omega^{(\ptca)}$, see definition \eqref{def:patches}.
We combine the single patch preconditioner, introduced in \eqref{def:single_patch_prec}, with an Additive Schwarz method.
Let us define a family of local spaces
\begin{align}\label{def:V_h_r}
	V_h^{(\ptca)}:= \spann \left\{  B_l :\, l \in \mathcal{J}^{(\ptca)} \right\}, & & \ptca = 1, \ldots , \NOmega,
\end{align}
with $\mathcal{J}^{(\ptca)}$ defined as in Section \ref{sec:mp}.
Therefore, $V_h^{(\ptca)}$ is the subspace of $V_h$ spanned by the B-splines basis function whose support intersect $\Omega^{(\ptca)}$.
Moreover, following the notation of \cite{toselli2006domain}, we consider restriction operators $R^{(\ptca)}: V_h \rightarrow V_h^{(\ptca)}$ with $\ptca = 1, \ldots , \NOmega$, defined by
\begin{align*}
	R^{(\ptca)} \left( \sum_{l \in \mathcal{J}} u_l B_l \right)= \sum_{l \in \mathcal{J}^{(\ptca)}} u_l B_l.
\end{align*}
Their transpose, in the basis representation, ${{R^{(\ptca)}}^T: V_h^{(\ptca)} \rightarrow V_h}$
correspond, in our case, to the inclusion of $V_h^{(\ptca)}$ into $V_h$. We denote with $\mathbf{R}^{(\ptca)}$ and $\mathbf{R}^{(\ptca)^T}$ the rectangular matrices associated to ${R^{(\ptca)}}$ and ${R^{(\ptca)}}^T$, respectively.
From now on, given $u^{(\ptca)} \in V_h^{(\ptca)}$, we will denote by $\vett[u]^{(\ptca)}$ the vector of its coordinates  with respect to the basis  $\{ B_l :\, l \in \mathcal{J}^{(\ptca)}\}$ and define the family of bilinear forms  
${a^{(\ptca)}: V_h^{(\ptca)} \times V_h^{(\ptca)} \rightarrow \mathbb{R}}$, for $\ptca=1, \ldots , \NOmega$, as
\begin{align}\label{def:loc_forms}
a^{(\ptca)}(u^{(\ptca)},v^{(\ptca)}):={\vett[v]^{(\ptca)}}^T\P^{(\ptca)} \vett[u]^{(\ptca)}, & & u^{(\ptca)},v^{(\ptca)} \in V_h^{(\ptca)},
\end{align}
with  
\begin{equation}\label{def:loc_forms2}
\P^{(\ptca)}:=	\Prsa,
\end{equation}
where we have set
\begin{align}\label{def:loc_matrices}
\begin{split}
[\widehat{\M}^{(\ptca)}]_{i,j}&:=\int_{\widehat{\Omega}} \widehat{B}^{(\ptca)}_{i}(\vett[x]) \widehat{B}^{(\ptca)}_{j}(\vett[x]) d \vett[x], \\ \widehat{\D}^{(\ptca)}&:=\text{diag}\left( \widehat{\M}^{(\ptca)} \right), \\
\D^{(\ptca)}&:=\text{diag}\left( \M^{(\ptca)} \right),
\end{split}
\end{align}
with the assumption that the basis functions $\{\widehat{B}^{(\ptca)}_{i}\}_{i=1}^{{\Ndof^{(\ptca)}}}$ and $\{{B}^{(\ptca)}_{i}\}_{i=1}^{\Ndof^{(\ptca)}}$ are ordered as described at the end of Section \ref{sec:mp}. We underline that the bilinear forms $\{a^{(\ptca)}\}$ are symmetric and positive definite.
The Additive Schwarz Preconditioner (inverse) is defined as
\begin{equation}\label{def:asp}
\P^{-1}_{\text{ad}}:= \sum_{{\ptca}=1}^{\NOmega} {\mathbf{R}^{(\ptca)}}^T {\P^{(\ptca)}}^{-1}\mathbf{R}^{(\ptca)}.
\end{equation}

The following Lemma follows straightforwardly from \cite[Theorem 2.7]{toselli2006domain} and provides a bound on the condition number of the multipatch mass matrix \eqref{eq:multipatch_mass} preconditioned by \eqref{def:asp}.
\begin{lem} \label{lem:widlund_toselli}
	Let the following three hypothesis be satisfied:
	\begin{itemize}
		\item (\emph{Stable Decomposition}) There exists a constant $C_{\mathrm{SD}} > 0$, such that every $u \in V_h$ admits a decomposition
\begin{align*}
		u=\sum_{{\ptca}=1}^{\NOmega} u^{(\ptca)}, & & \text{ with } u^{(\ptca)} \in V_h^{(\ptca)},
\end{align*}
		that satisfies
\begin{equation*}
		\sum_{{\ptca}=1}^{\NOmega} a^{(\ptca)}( u^{(\ptca)},u^{(\ptca)} ) \leq C_{\mathrm{SD}} \| u\|^2_{L^2(\Omega)}.
\end{equation*}
		\item (\emph{Strengthened Cauchy-Schwarz Inequalities}) There exist constants $0 \leq \epsilon_{{\ptca}{\ptcb}} \leq 1$, for $1 \leq {\ptca},{\ptcb} \leq \NOmega$, such that
\begin{equation*}
		\lvert (  u^{(\ptca)}, u^{(\ptcb)} )_{L^2(\Omega)} \rvert \leq \epsilon_{{\ptca}{\ptcb}} \|  u^{(\ptca)} \|_{L^2(\Omega)} \|  u^{(\ptcb)} \|_{L^2(\Omega)},
\end{equation*}
		for $u^{(\ptca)} \in V_h^{(\ptca)}$ and $u^{(\ptcb)} \in V_h^{(\ptcb)}$.
		\item (\emph{Local Stability}) There exists $C_{\mathrm{LS}} >0$, such that for all ${\ptca}=1, \ldots , \NOmega$,
\begin{align*}
		\| u^{(\ptca)} \|^2_{L^2(\Omega)} \leq C_{\mathrm{LS}} a^{(\ptca)}( u^{(\ptca)},u^{(\ptca)} ), & & \forall \,u^{(\ptca)} \in 
		V_h^{(\ptca)}.
\end{align*}
	\end{itemize}
	Then the condition number of the preconditioned operator satisfies
\begin{equation*}
	\kappa\left(\PadMPad \right) \leq C_{\mathrm{SD}}  C_{\mathrm{LS}}\rho( \mathcal{E} ),
\end{equation*}
	where $\rho \left( \mathcal{E} \right)$ represents the spectral radius of the matrix $\mathcal{E}=\left\lbrace \epsilon_{{\ptca}{\ptcb}} \right\rbrace$.
\end{lem}

We are now able to present the main result of this section.
\begin{thm}\label{thm:additive_mp}
Under Assumption \ref{ass:multipatch}, there exists a constant $C$,  independent of $h$ and $\Npatch$, verifying
\begin{align*}
	\kappa \left( \PadMPad \right) \leq C \Npatch^2,
\end{align*}
where $\Npatch$, defined in \eqref{def:N_patch}, denotes the maximum number of adjacent patches.
\end{thm}
\begin{proof}
	We show that the hypothesis of Lemma \ref{lem:widlund_toselli} hold with $C_{\mathrm{SD}}$ and $C_{\mathrm{LS}}$ independent of $h$ and $\rho(\mathcal{E}) \leq \Npatch$.
\paragraph*{Part I: Stable Decomposition.} The argument we use  is similar to the one presented in \cite[Lemma 4.1]{Pavarino2012}. 
	Given $u \in V_{h}$, with
\begin{align*}
	u= \sum_{l \in  \mathcal{J}} u_l B_l,
\end{align*}
	we define $u^{(\ptca)} \in V_{h}^{({\ptca})}$ as
\begin{align*}
	u^{(\ptca)}:=\sum_{l \in \mathcal{J}^{(\ptca)}} \frac{u_l}{n_l} B_l, & & {\ptca}=1 ,\ldots , \NOmega,
\end{align*}
where $n_l$ is defined in \eqref{def:n_l}. It is straightforward to see that
\begin{align*}
	\sum_{{\ptca}=1}^{\NOmega} u^{(\ptca)} = u.
\end{align*}
	Recalling definitions \eqref{def:loc_forms}, \eqref{def:loc_forms2} and \eqref{def:loc_matrices} and introducing
	\begin{align*}
		[\M^{(\ptca)}]_{i,j}:=\int_{\Omega^{(\ptca)}} {B}^{(\ptca)}_{i}(\vett[x]) {B}^{(\ptca)}_{j}(\vett[x]) d \vett[x],
	\end{align*}
	we have
\begin{align}\label{eq:C_sd_1}
\begin{split}
	\sum_{{\ptca}=1}^{\NOmega} a^{(\ptca)}( \ua,\ua )  &= \sum_{{\ptca}=1}^{\NOmega} {\vett[u]^{(\ptca)}}^T\P^{(\ptca)} \vett[u]^{(\ptca)} \\
	&= \sum_{{\ptca}=1}^{\NOmega} \frac{{\vett[u]^{(\ptca)}}^T\P^{(\ptca)} \vett[u]^{(\ptca)}}{{\vett[u]^{(\ptca)}}^T\M^{(\ptca)} \vett[u]^{(\ptca)}} {\vett[u]^{(\ptca)}}^T\M^{(\ptca)} \vett[u]^{(\ptca)} \\
	&\leq \sum_{{\ptca}=1}^{\NOmega} \frac{{\vett[u]^{(\ptca)}}^T\P^{(\ptca)} \vett[u]^{(\ptca)}}{{\vett[u]^{(\ptca)}}^T\M^{(\ptca)} \vett[u]^{(\ptca)}} \| \ua_{|\Omega^{(\ptca)}} \|^2_{L^2(\Omega^{(\ptca)})} \\
	&\leq \left(\max_{{\ptca}=1, \ldots , \NOmega} \frac{{\vett[u]^{(\ptca)}}^T\P^{(\ptca)} \vett[u]^{(\ptca)}}{{\vett[u]^{(\ptca)}}^T\M^{(\ptca)} \vett[u]^{(\ptca)}}\right) \sum_{{\ptca}=1}^{\NOmega} \| \ua_{|\Omega^{(\ptca)}} \|^2_{L^2(\Omega^{(\ptca)})}.
	\end{split}
\end{align}
Combining Corollary \ref{cor:cond_bound_phys} and Corollary \ref{cor:cond_bound_prec}, we obtain a constant $C_{\mathrm{max}}$, independent of $h$, such that
\begin{align}\label{eq:est_ratio_1}
	\frac{{\vett[u]^{(\ptca)}}^T\P^{(\ptca)} \vett[u]^{(\ptca)}}{{\vett[u]^{(\ptca)}}^T\M^{(\ptca)} \vett[u]^{(\ptca)}} \leq  C_{\mathrm{max}}, & & \forall u^{(\ptca)} \in V_h^{(\ptca)},\, \forall \ptca =1, \ldots, \NOmega.
\end{align}
By observing that $n_l \geq 1$, for all $l$, it follows
\begin{align*}
	\left\lvert \frac{u_l}{n_l} \right\rvert \leq \left\lvert u_l \right\rvert 
\end{align*}
and thus 
\begin{align}\label{eq:ineq_coeff}
\left\lvert \varphi_{\mathbf{i}}(u^{(\ptca)}_{|\Omega^{(\ptca)}}) \right\rvert \leq \left\lvert \varphi_{\mathbf{i}}(u_{|\Omega^{(\ptca)}}) \right\rvert, & & \forall \mathbf{i} \in \mathbf{I}^{(\ptca)},
\end{align}
where $\{\varphi_{\mathbf{i}}\}_{\mathbf{i} \in \mathbf{I}^{(\ptca)}}$ denotes the dual basis introduced in Corollary \eqref{cor_dualbasis_physical}, relative to the isogeometric space defined on the  patch $\Omega^{(\ptca)}$.
Using Corollary \ref{cor_dualbasis_physical}, \eqref{eq:ineq_coeff} and \eqref{eq:est_2_bis} in the patch $\Omega^{(\ptca)}$ and the adjacent ones $\Omega^{(\ptcb)}$, there exist constants $C_1,C_2,C_3$, independent of $h$ and $\Npatch$, such that
\begin{align*}
\| u^{(\ptca)} \|^2_{L^2(\Omega^{(\ptca)})} & \leq C_1  h^d \sum_{\mathbf{i} \in \mathbf{I}^{(\ptca)}} \varphi_{\mathbf{i}}(u^{(\ptca)})^2 \leq C_1  h^d \sum_{\mathbf{i} \in \mathbf{I}^{(\ptca)}} \varphi_{\mathbf{i}}(u_{|\Omega^{(\ptca)}})^2 \\
& \leq C_2  \sum_{\mathbf{i} \in \mathbf{I}^{(\ptca)}} \|         u_{|\Omega^{(\ptca)}} \|^2_{L^2(\vett[F]^{(\ptca)}(Q'_{\mathbf{i}}))} \leq C_3  \| u  \|^2_{L^2(\Omega^{(\ptca)})}
\end{align*}
Finally, summing over all $\ptca \in \{ 1, \ldots , \NOmega \}$, it holds
\begin{align}\label{eq:C_sd_3}
\sum_{{\ptca}=1}^{\NOmega} \| \ua \|^2_{L^2(\Omega^{(\ptca)})} \leq C_3 \| u \|^2_{L^2(\Omega)}.
\end{align}
Combining  \eqref{eq:C_sd_1}, \eqref{eq:est_ratio_1} and \eqref{eq:C_sd_3}, we obtain
\begin{align}\label{eq:estimate_C_0}
C_{\mathrm{SD}} \leq C_3 C_{\mathrm{max}}.
\end{align}
\paragraph*{Part II: Strengthened Cauchy-Schwarz Inequalities.}
Standard Cauchy-Schwarz inequality, assures us that $\epsilon_{{\ptca}{\ptcb}} \leq 1$, for all $1 \leq {\ptca},{\ptcb} \leq \NOmega$. Furthermore, for each $\ptca \in \{1 , \ldots \NOmega \}$, there are at most $\Npatch$ indices $\ptcb \in \{1 , \ldots \NOmega \}$ such that there exists two basis functions $B_{l_1} \in \{ B_l :\, l \in \mathcal{J}^{(\ptca)}\}$ and \linebreak ${B_{l_2} \in \{ B_l :\, l \in \mathcal{J}^{(\ptcb)}\}}$ with $\supp(B_{l_1}) \cap \supp(B_{l_2}) \neq \emptyset$.
	As a consequence, in every row of the matrix $\mathcal{E}=\left\lbrace \epsilon_{{\ptca}{\ptcb}} \right\rbrace$ there are at most $\Npatch$ non-zero entries. 
	Combining these facts, we can conclude that the spectral radius of $\mathcal{E}$ satisfies:
\begin{equation}\label{eq:estimate_rho}
	\rho \left( \mathcal{E} \right) \leq  \Npatch.
\end{equation}

\paragraph*{Part III: Local Stability.}
Using \eqref{eq:est_2_bis} in the patch $\Omega^{(\ptca)}$ and the adjacent ones $\Omega^{(\ptcb)}$, there exists a constant $C$, independent of $h$ and $\Npatch$, such that
\begin{align*}
	\| u^{(\ptca
		)}\|^2_{L^2(\Omega)}& = \sum_{\substack{{\ptcb} \,: \overline{\Omega^{({\ptca})}} \cap \overline{\Omega^{({\ptcb})}} \neq \emptyset}} \| u^{(\ptca
		)}_{|\Omega^{(\ptcb)}}\|^2_{L^2(\Omega^{(\ptcb)})} \\
	&\leq C \Npatch h^d \sum_{\mathbf{i} \in \mathbf{I}^{(\ptca)}} (u^{(\ptca)}_{\mathbf{i}})^2   = C \Npatch h^d (\vett[u]^{(\ptca)})^T \vett[u]^{(\ptca)}.
\end{align*}
It holds
\begin{align}\label{eq:star_star}
\begin{split}
h^d (\vett[u]^{(\ptca)})^T \vett[u]^{(\ptca)} & = h^d \frac{(\vett[u]^{(\ptca)})^T (\vett[u]^{(\ptca)})}{(\vett[u]^{(\ptca)})^T \P^{(\ptca)} \vett[u]^{(\ptca)}} (\vett[u]^{(\ptca)})^T \P^{(\ptca)} \vett[u]^{(\ptca)} \\
& \leq \frac{h^d}{\lambda_{\mathrm{min}}(\P^{(\ptca)})} a^{(\ptca)}(u^{(\ptca)},u^{(\ptca)}).
\end{split}
\end{align}
Using \eqref{eq:star_star} and Corollary \ref{cor:cond_bound_prec}, finally yields local stability:
\begin{equation}\label{eq:loc_stab}
	\| u^{(\ptca)}\|^2_{L^2(\Omega)} \leq C \Npatch a^{(\ptca)}(u^{(\ptca)},u^{(\ptca)}),
\end{equation}
with another constant $C$, independent of $h$ and $\Npatch$.

Finally, by applying Lemma \ref{lem:widlund_toselli} with the estimates provided in \eqref{eq:estimate_C_0}, \eqref{eq:estimate_rho} and \eqref{eq:loc_stab}, we obtain that there exists a constant $C$, independent of $h$, verifying
\begin{align*}
		\kappa\left( \PadMPad \right) \leq C \Npatch^2.
\end{align*}
\end{proof}

\section{Preconditioners application and
  cost}\label{sec:application_cost}
The mass matrices and the preconditioners introduced in this paper are symmetric and positive definite. We then adopt the Preconditioned Conjugate Gradient method (PCG) to solve the associated linear systems.
For evaluating the computational cost of PCG, we recall that for each iteration, the two most expensive steps are: the solution of a linear system associated to the preconditioner and the computation of the residual, through a matrix-vector with $\M$.
We recall that all the univariate matrices have dimension $m$. Then, the single patch mass matrix has dimension $\Ndof=m^d$, while for the multipatch one we have $\Ndof \approx \NOmega m^d$.

\subsection{Single patch preconditioner}
The application of the single patch preconditioner is the solution of a linear system associated to
\begin{align*}
	\P=\Prs.
\end{align*}
Thanks to \eqref{eq:kron_prod}, it holds
\begin{align*}
	\widehat{\D}= \text{diag}(\Mp_d \otimes \dots \otimes \Mp_1)= \widehat{\D}_d \otimes \dots \otimes \widehat{\D}_1,
\end{align*}
where we have set $\widehat{\D}_i=\text{diag}(\Mp_i)$, for $i=1, \ldots , d$, and
\begin{align*}
	\widehat{\D}^{-\frac{1}{2}} \Mp \widehat{\D}^{-\frac{1}{2}} = \widehat{\D}_d^{-\frac{1}{2}} \Mp_d \widehat{\D}_d^{-\frac{1}{2}} \otimes \dots \otimes \widehat{\D}_1^{-\frac{1}{2}} \Mp_1 \widehat{\D}_1^{-\frac{1}{2}}.
\end{align*}
By exploiting  \eqref{eq:kron_inv}, the inverse of $\P$ may be expressed as
\begin{align*}
	\P^{-1}& =\left( \D^{\frac{1}{2}} \widehat{\D}^{-\frac{1}{2}} \Mp \widehat{\D}^{-\frac{1}{2}} \D^{\frac{1}{2}} \right)^{-1} \\
	& = \D^{- \frac{1}{2}}\left( \widehat{\D}_d^{-\frac{1}{2}} \Mp_d \widehat{\D}_d^{-\frac{1}{2}} \right)^{-1} \otimes \dots \otimes \left( \widehat{\D}_1^{-\frac{1}{2}} \Mp_1 \widehat{\D}_1^{-\frac{1}{2}} \right)^{-1} \D^{-\frac{1}{2}}.
\end{align*}
Therefore, the solution of a linear system associated to $\P$ can be summarized as follows.
\begin{algorithm}
	\caption{Single patch }\label{al:single_patch}
	\begin{algorithmic}[1]
		\State Assemble the matrices $\widehat{\D}_i^{-\frac{1}{2}} \Mp_i \widehat{\D}_i^{-\frac{1}{2}}$, for $i=1,\ldots , d$. 
		\State Compute the diagonal scaling $\vett[\widetilde{z}]=\D^{-\frac{1}{2}} \vett[z]$.
		\State Solve the linear system $\left(\widehat{\D}_d^{-\frac{1}{2}} \Mp_d\widehat{\D}_d^{-\frac{1}{2}} \otimes \dots \otimes \widehat{\D}_1^{-\frac{1}{2}} \Mp_1 \widehat{\D}_1^{-\frac{1}{2}} \right) \vett[\widetilde{y}]= \vett[\widetilde{z}]$.
		\State Compute the diagonal scaling $\vett[y]=\D^{-\frac{1}{2}} \vett[\widetilde{y}]$.
	\end{algorithmic}
\end{algorithm}

Step 1 represents the preconditioner setup. The matrices
$\widehat{\D}_i^{-\frac{1}{2}} \Mp_i \widehat{\D}_i^{-\frac{1}{2}}$
need to be constructed only once, before starting the PCG solver. The
overall cost of this step is $C(p)dm $ FLOPs,
where $C(p)$ denotes a constant that depends on $p$  and depends on
how the matrices $\Mp_i$ are computed: Gauss quadrature is the least
efficient approach and in such a case $C(p) = O(p^3)$. However this
cost can be considered negligible in practice (for examples, in all
the tests we present in Section \ref{sec:tests}, where  $p \ll m$), since  $m = \Ndof
^{1/d}$ and   Steps 2-4 have a cost which is proportional to $\Ndof$. Furthermore,  \B 
 Steps 2-4 need to be performed at each iteration. 
Both Steps 2 and 4 consist in the product of a diagonal matrix by a vector, thus their cost is $2\Ndof$ FLOPs.  
Thanks to  \eqref{eq:kronecker_vec_prod} and recalling that univariate mass matrices are symmetric banded matrices with bandwidth $p$,  Step 3 costs roughly $2d(2p+1) \Ndof = O(p \Ndof)$ FLOPs.
To sum up, we get that the application of Algorithm \ref{al:single_patch} 
requires roughly $2 \left( d(2p+1) + 1 \right) \Ndof = O(p \Ndof )$ FLOPs. 
We emphasize that the cost of our preconditioner  is proportional to $\Ndof$, and depends linearly with respect to $p$. Moreover, this costs is even smaller than that required for the residual computation PCG (or any iterative solver). Indeed, having in mind that the computational cost of a matrix-vector product is twice the number of non-zero entries of that matrix and that for the isogeometric mass matrix this number is at most $(2p+1)^d \Ndof$,  it follows that the residual computation requires $2(2p+1)^d \Ndof = O(p^d \Ndof)$ FLOPs. 

\subsection{Multipatch preconditioner}
The application of $\Prad$, provided in \eqref{def:asp}, involves, for $\ptca \in \{ 1, \ldots , \NOmega \}$, the application of the operators $R^{(\ptca)}$ and  ${R^{(\ptca)}}^T$, whose cost is negligible, and the application of $\left({\P^{(\ptca)}}\right)^{-1}$, whose cost has been analyzed in the previous section. 
 In conclusion, the cost of application of $\Prad$ is $O( p \sum_{{\ptca}=1}^{\NOmega}\Ndof^{(\ptca)})=O(p\Ndof)$. 

\section{Numerical Tests}\label{sec:tests}
In this section we show the performance of the preconditioners presented in this paper.
In our simulations, we consider only sequential executions and we force the use of a single computational thread in a Intel Core i7-5820K processor, running at 3.30 GHz and with 64 GB of RAM. All the tests are performed with Matlab R2015a and GeoPDEs toolbox \cite{Vazquez2016}. The linear system is solved by PCG, with tolerance equal to $10^{-8}$ and with the null vector as initial guess.
We denote by $n_{\mathrm{sub}}$ the number of subdivisions, which are the same in each parametric direction and in each patch. Moreover, we underline that we only consider splines of maximal regularity. The symbol ``*'' denotes the impossibility of formation of the matrix $\M$, due to memory requirements.

For assessing the performance of the preconditioners, we consider the problem of finding the $L^2$-projection of a given function $f$, on different domains, see Figures \ref{fig:2D} and \ref{fig:3D}. For bidimensional problems, the given function is \linebreak
${f(x,y)=\cos(\pi x) \cos(\pi y)}$, while for the tridimensional ones, we have set \linebreak ${f(x,y,z)=\cos(\pi x) \cos(\pi y) \cos(\pi z)}$.

\begin{figure}
	\begin{subfigure}{0.49\textwidth}
		\includegraphics[width=\textwidth]{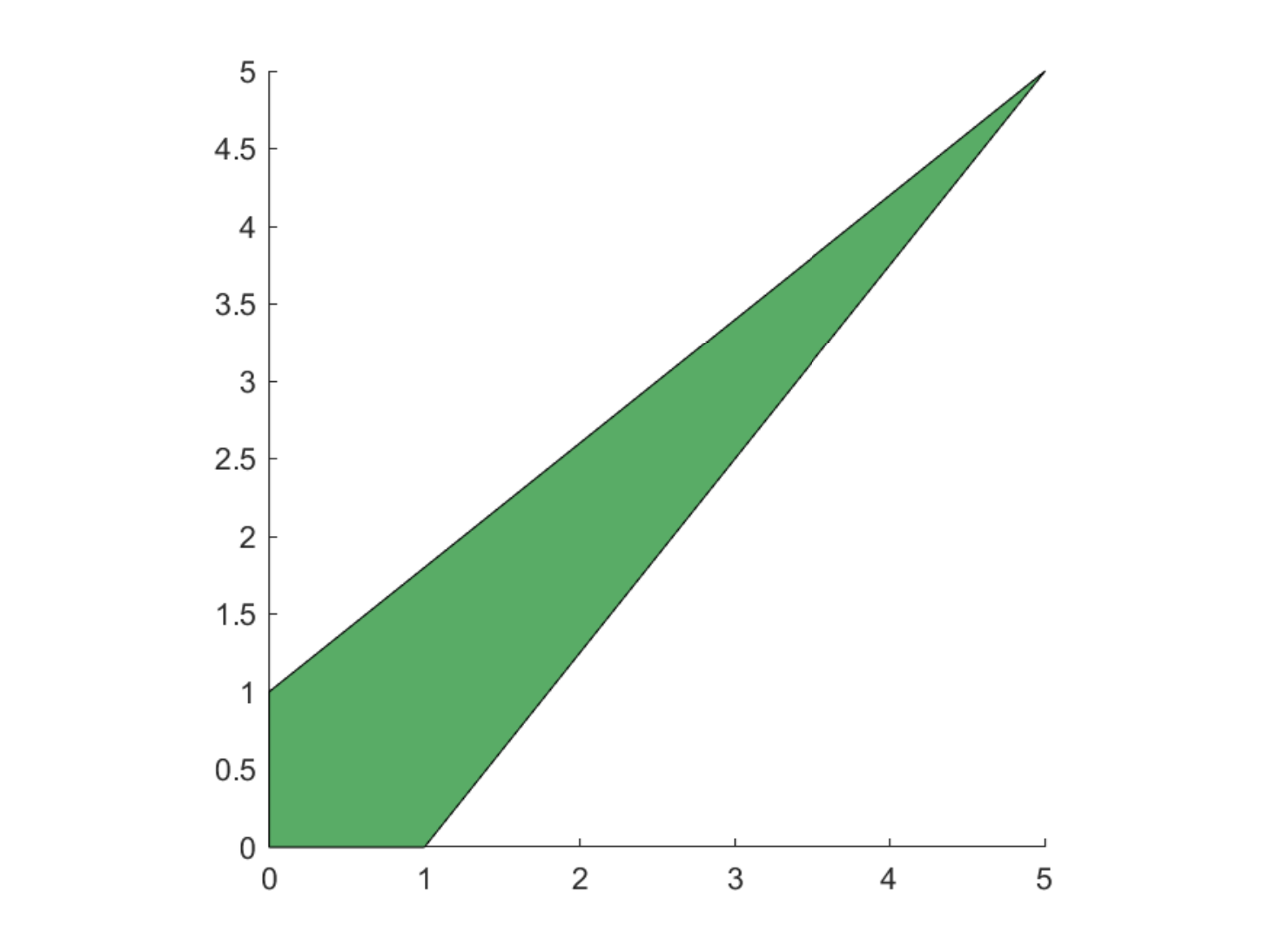}
		\caption{Kite.}
		\label{fig:Kite}
	\end{subfigure}
	\begin{subfigure}{0.49\textwidth}
		\includegraphics[width=\textwidth]{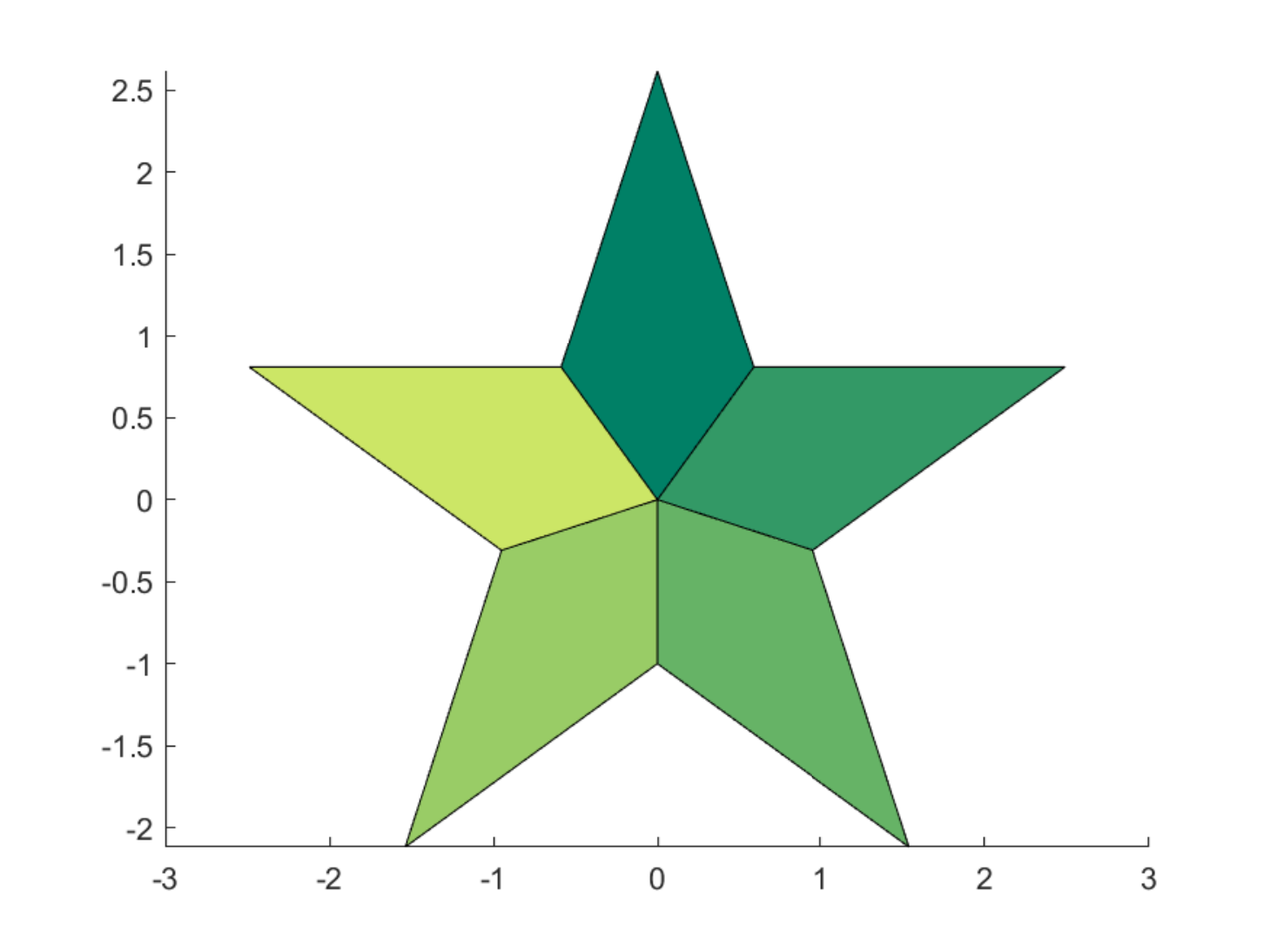}
		\caption{Multipatch Star.}
		\label{fig:Star_mp}
	\end{subfigure}
	\begin{subfigure}{0.49\textwidth}
		\includegraphics[width=\textwidth]{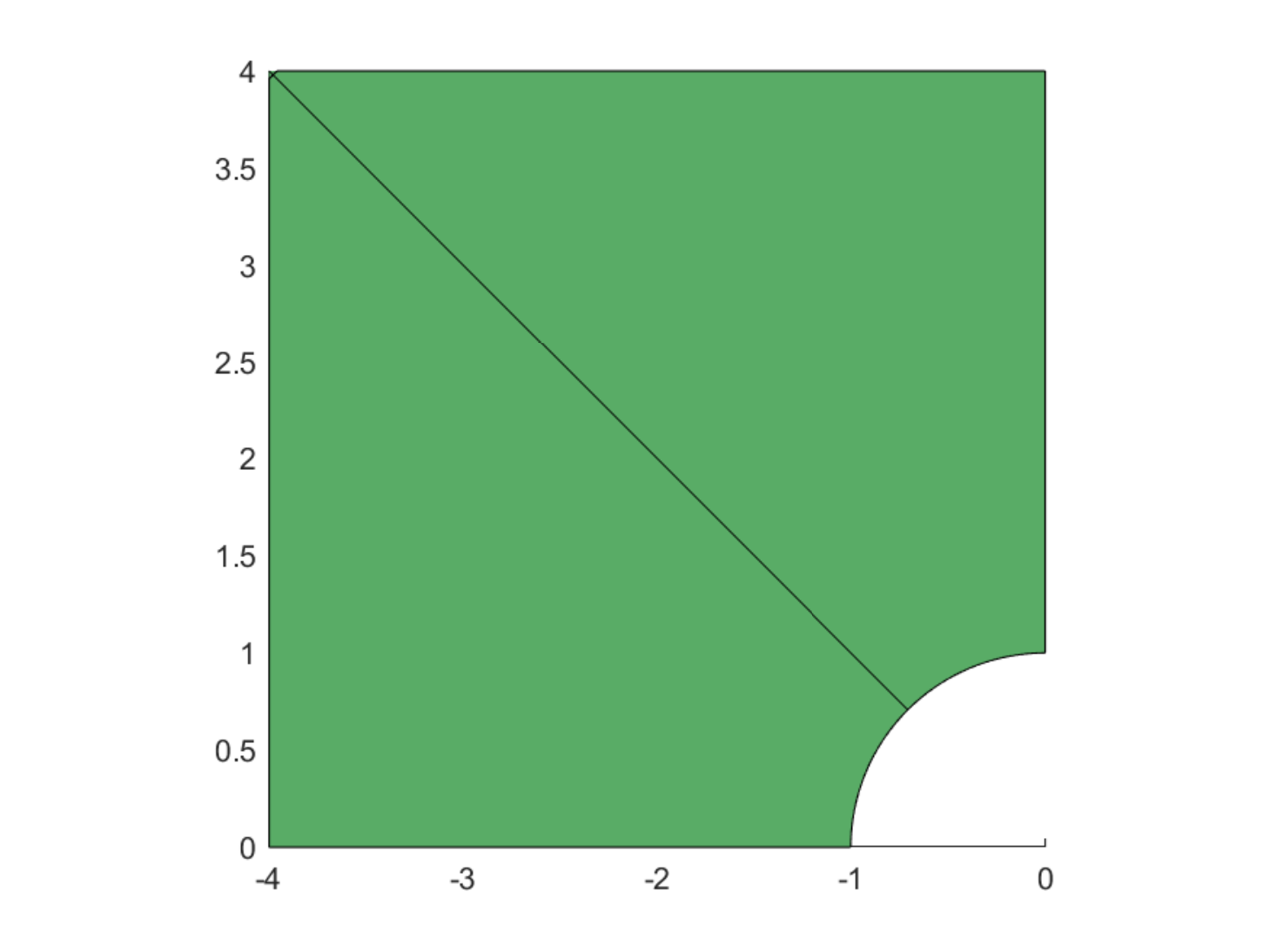}
		\caption{Holed plate.}
		\label{fig:Hollow}
	\end{subfigure}
	\begin{subfigure}{0.49\textwidth}
		\includegraphics[width=\textwidth]{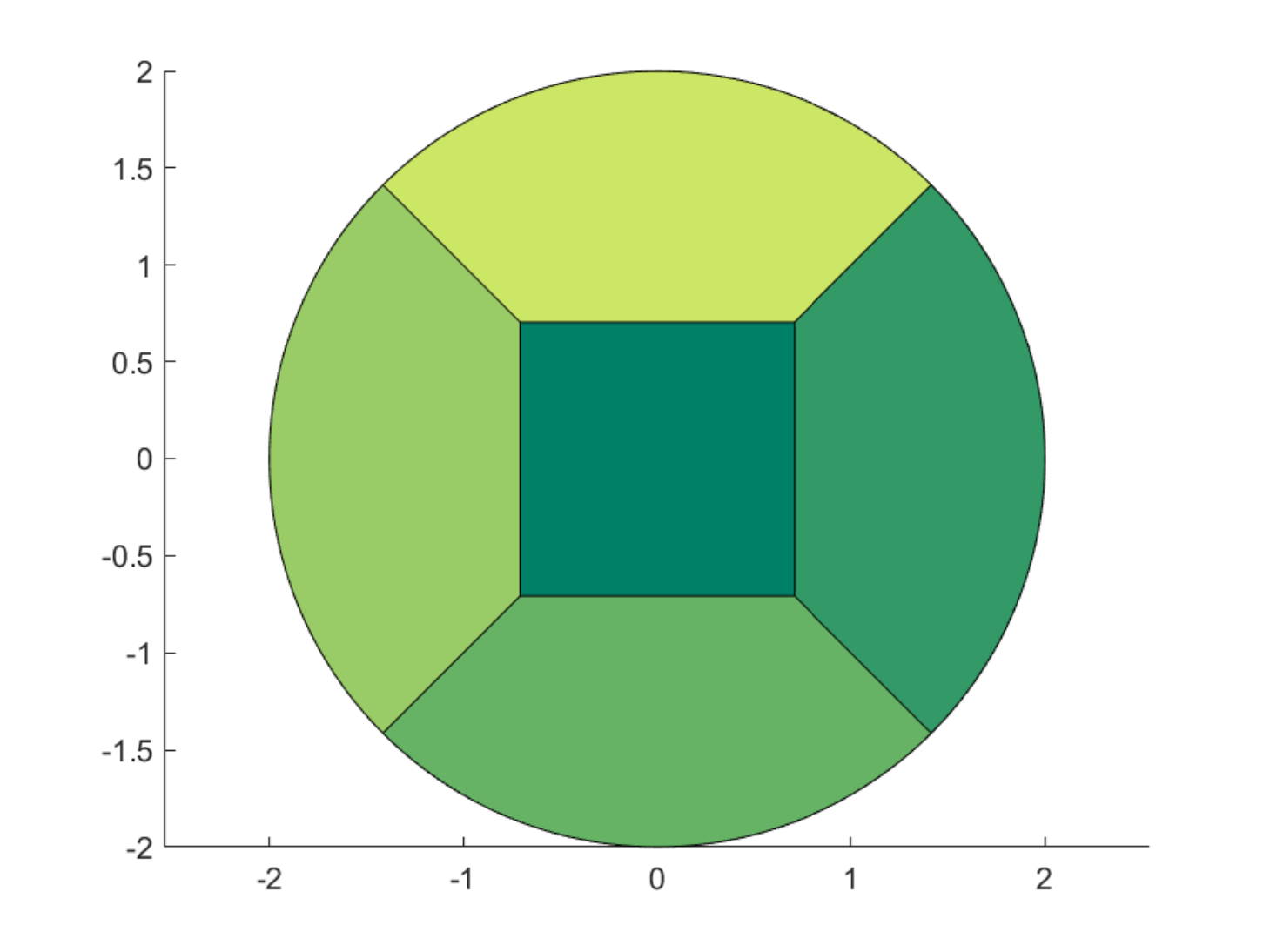}
		\caption{Multipatch Disc.}
		\label{fig:Disc_mp}
	\end{subfigure}
\begin{subfigure}{0.49\textwidth}
	\includegraphics[width=\textwidth]{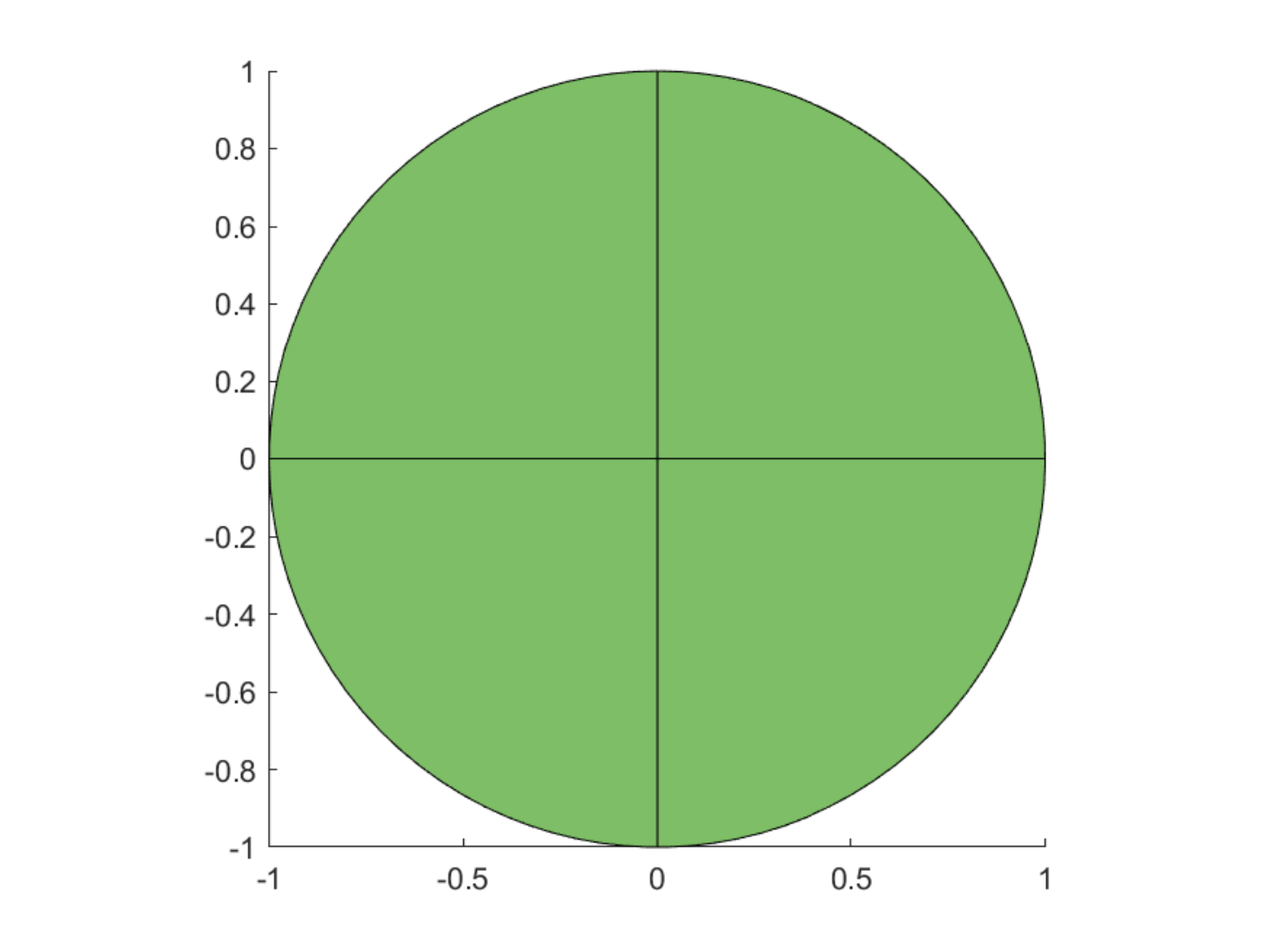}
	\caption{Disc with one singularity.}
	\label{fig:Disc1}
\end{subfigure}
\begin{subfigure}{0.49\textwidth}
	\includegraphics[width=\textwidth]{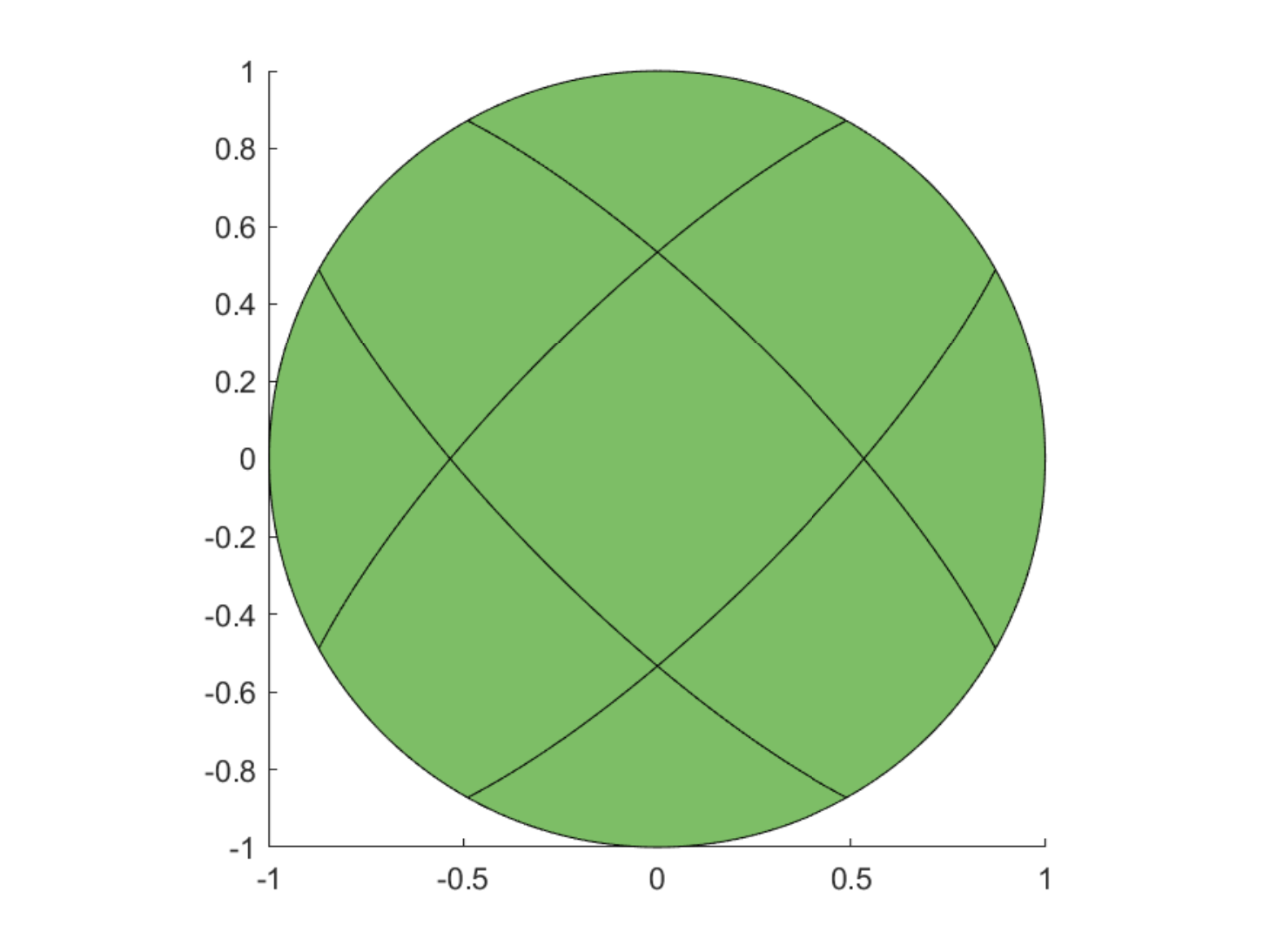}
	\caption{Disc with four singularities.}
	\label{fig:Disc4}
\end{subfigure}
\caption{Bidimensional domains.}
\label{fig:2D}
\end{figure}

\begin{figure}
	\begin{subfigure}{0.49\textwidth}
		\includegraphics[width=\textwidth]{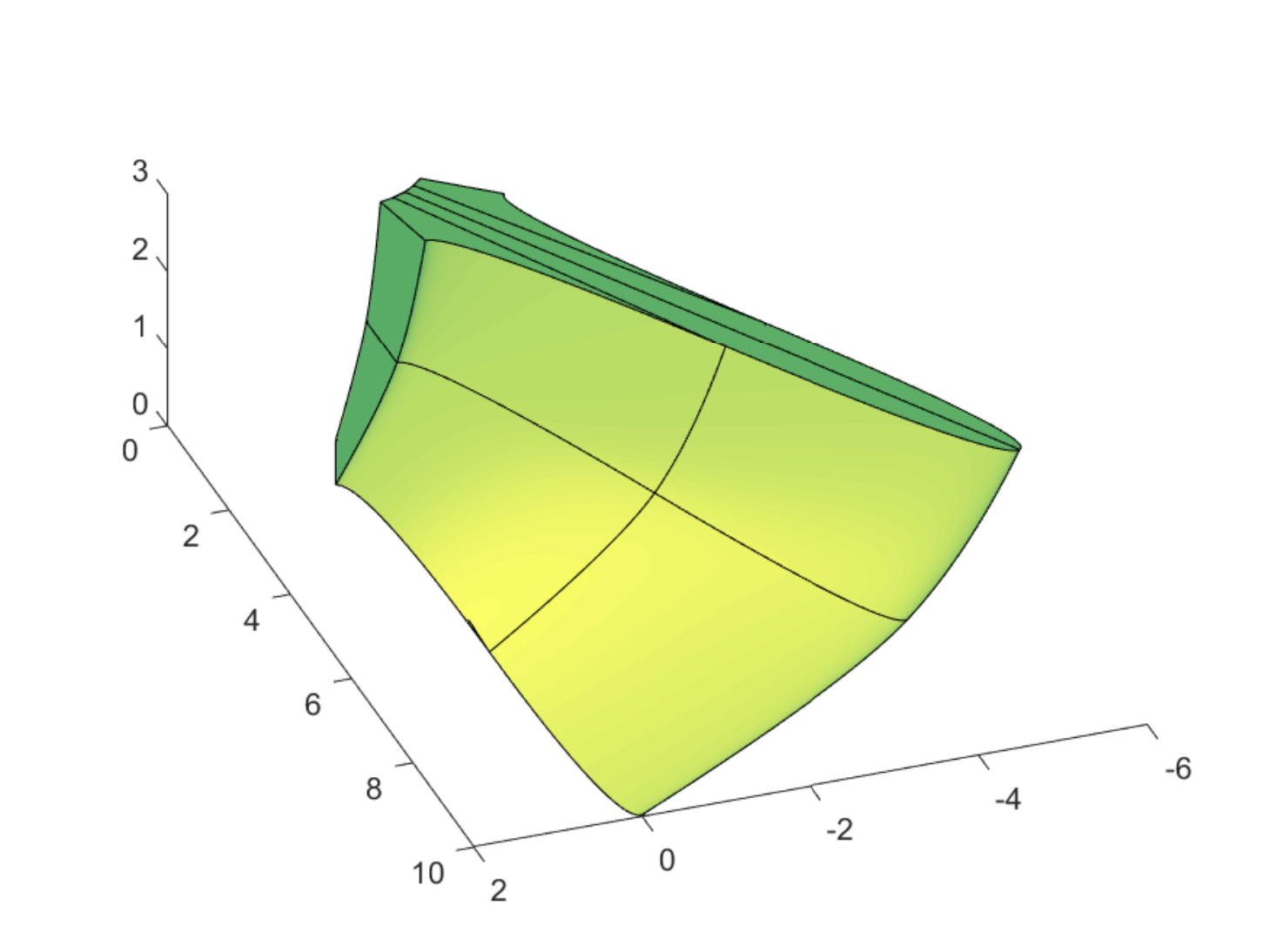}
		\caption{Blade.}
		\label{fig:Blade}
	\end{subfigure}
	\begin{subfigure}{0.49\textwidth}
		\includegraphics[width=\textwidth]{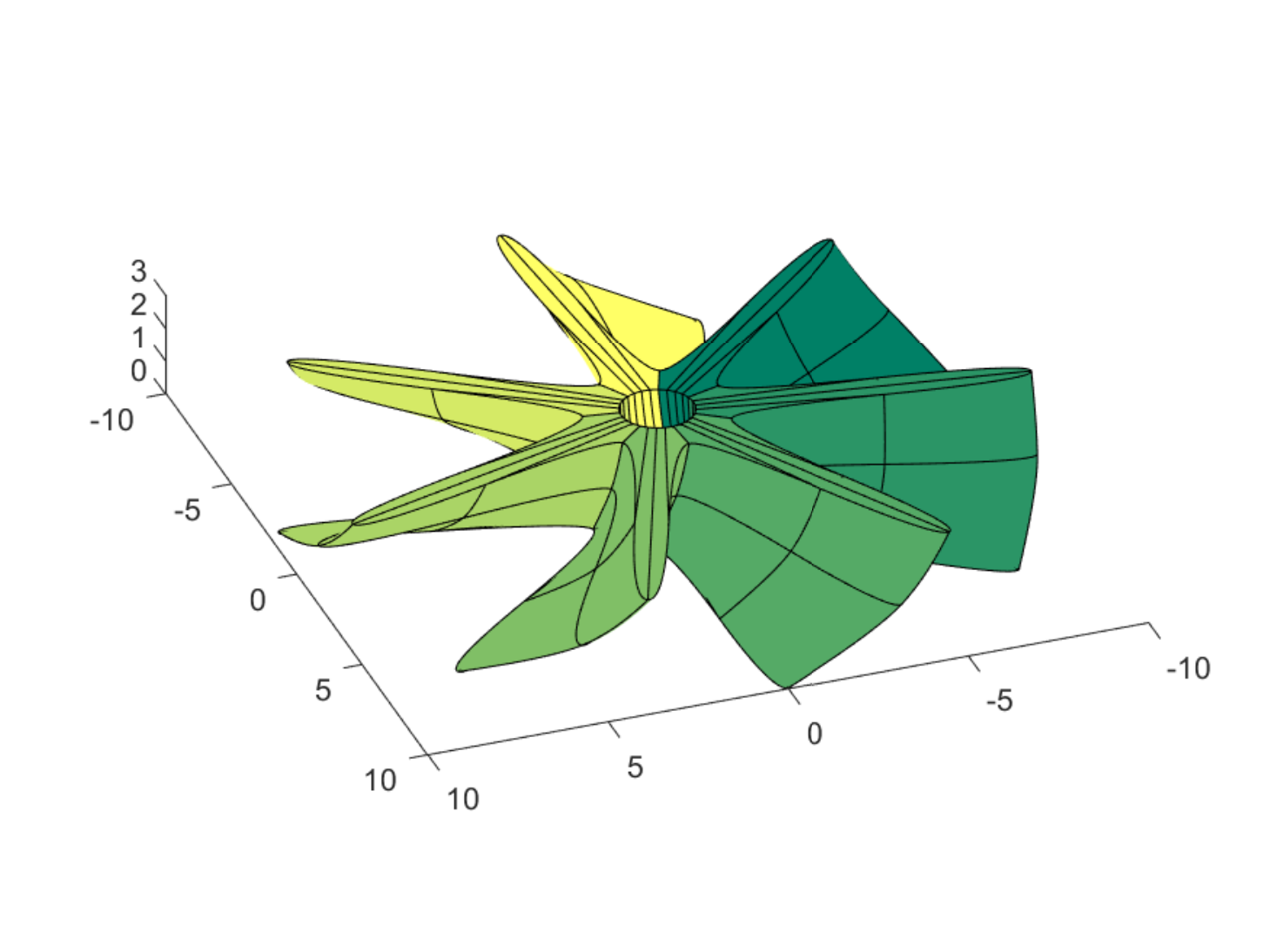}
		\caption{Multipatch Fan.}
		\label{fig:Fan_mp}
	\end{subfigure}
	\caption{Tridimensional domains.}
	\label{fig:3D}
\end{figure}

\subsection{Single Patch domains}
As examples of regularly parametrized single patch domains, we consider a bidimensional kite and a tridimensional blade (see Figures \ref{fig:Kite} and \ref{fig:Blade}).
For the kite domain, we compute the condition number of the unpreconditioned and preconditioned mass matrix for different values of $h$ and $p$ and report them in Tables \ref{table:Kite_cond_ump}  and \ref{table:Kite_cond}, respectively. By comparing these numbers, we can see that the condition number is dramatically reduced by our preconditioning strategy. In particular, as predicted by Theorem \ref{thm:single_patch_h}, the condition number of preconditioned matrices converges to 1 as the mesh-size $h$ goes to 0. 
Tables \ref{table:Kite_its} and \ref{table:Blade_its} show the number of iterations and computation time spent by PCG for the kite and the blade domain, respectively. We emphasize that the number of iterations is always very low and even decreases when $h$ is reduced.

\renewcommand\arraystretch{1.1} 
\begin{table}
	\centering                                                                            
	\begin{tabular}{|c|c|c|c|c|c|}                   
		\hline                                           
		$n_{\mathrm{sub}}$ & $p=2$ & $p=3$ & $p=4$ & $p=5$ & $p=6$ \\ 
		\hline                                                               
		16 & $5.540 \cdot 10^2$ & $2.980 \cdot 10^3$ & $ 1.673 \cdot 10^4$ & $9.892 \cdot 10^4$ & $ 6.106 \cdot 10^5$ \\  
		\hline                                                               
		32 & $7.040 \cdot 10^2$ & $ 4.063 \cdot 10^3$ & $2.435 \cdot 10^4$ & $ 1.523 \cdot 10^5$ & $ 9.853 \cdot 10^5$ \\ 
		\hline                                                               
		64 & $8.150 \cdot 10^2$ & $4.929 \cdot 10^3$ & $3.082  \cdot 10^4$ & $ 2.002 \cdot 10^5$ & $1.340 \cdot 10^6$ \\
		\hline
		128 & $8.900 \cdot 10^2$ & $5.536 \cdot 10^3$ & $3.555  \cdot 10^4$ & $ 2.366 \cdot 10^5$ & $1.617 \cdot 10^6$ \\
		\hline                                                               
	\end{tabular}
	\caption{Condition number of mass matrix for kite.} 
	\label{table:Kite_cond_ump} 
\end{table}   

\renewcommand\arraystretch{1.1} 
\begin{table}
\centering                                                                            
\begin{tabular}{|c|c|c|c|c|c|}                   
\hline                                           
$n_{\mathrm{sub}}$ & $p=2$ & $p=3$ & $p=4$ & $p=5$ & $p=6$ \\ 
\hline                                            
16 & 1.056 & 1.077 & 1.103 & 1.129 & 1.157 \\ 
\hline                                            
32 & 1.034 & 1.047 & 1.062 & 1.078 & 1.094 \\ 
\hline                                            
64 & 1.019 & 1.027 & 1.035 & 1.045 & 1.054 \\ 
\hline                                            
128 & 1.010 & 1.015 & 1.019 & 1.024 & 1.030 \\
\hline      
\end{tabular} 
\caption{Condition number of preconditioned mass matrix for kite.} 
\label{table:Kite_cond} 
\end{table}   

\renewcommand\arraystretch{1.1} 
\begin{table}
\centering                                  
\begin{tabular}{|c|c|c|c|c|c|}                   
\hline                                           
$n_{\mathrm{sub}}$ & $p=2$ & $p=3$ & $p=4$ & $p=5$ & $p=6$ \\ 
\hline                                                                     
16 & 4 / 0.00134 & 4 / 0.00141 & 4 / 0.00152 & 4 / 0.00167 & 4 / 0.00184 \\
\hline                                                                     
32 & 3 / 0.00191 & 3 / 0.00203 & 3 / 0.00225 & 4 / 0.00316 & 4 / 0.00352 \\
\hline                                                                     
64 & 3 / 0.00465 & 3 / 0.00525 & 3 / 0.00578 & 3 / 0.00675 & 3 / 0.00812 \\
\hline                                                                     
128 & 3 / 0.0155\z & 3 / 0.0181\z & 3 / 0.0213\z & 3 / 0.0255\z & 3 / 0.0310\z \\    
\hline
\end{tabular}
\caption{Iterations and time spent by PCG for kite.}
\label{table:Kite_its} 
\end{table}

\renewcommand\arraystretch{1.1} 
\begin{table}
	\centering                                  
	\begin{tabular}{|c|c|c|c|c|c|}                   
		\hline                                           
		$n_{\mathrm{sub}}$ & $p=2$ & $p=3$ & $p=4$ & $p=5$ & $p=6$ \\ 
		\hline                                                                      
		16 & 6 / 0.0137 & 6 / 0.0315 & 6 / 0.0654 & 6 / 0.137 & 7 / 0.274 \\        
		\hline                                                                      
		32 & 5 / 0.0821 & 5 / 0.154\z & 5 / 0.299\z & 5 / 0.608 & 6 / 1.27\z \\           
		\hline                                                                      
		64 & 4 / 0.499\z & 4 / 1.01\z\z & 4 / 1.75\z\z & 4 / 3.49\z & 4 / 6.37\z \\ 
		\hline                                                            
	\end{tabular}
	\caption{Iterations and time spent by PCG for blade.}
	\label{table:Blade_its} 
\end{table}

The case of singularly parametrized domains is beyond the theory of Section \ref{sec:single_patch} (Assumption \ref{ass:g_regularity} does not hold). Nevertheless, we test numerically this situation on three examples: a holed plate with a singular point in the top left vertex (see Figure \ref{fig:Hollow}); a disc with a singularity in the center (Figure \ref{fig:Disc1}) and a disc with four singularities on the boundary (Figure \ref{fig:Disc4}). 
In all the three examples the condition number is always close to 1 and, even though it does not converge to 1 as in the non-singular case, it does not grow as $h$ goes to 0. 
Accordingly, the number of PCG iterations is very low (see Tables \ref{table:Hollow_its}, \ref{table:Disc1_its}  and \ref{table:Disc4_its}).

We are interested in studying the dependence on $p$ of the condition number of the preconditioned system. 
For this purpose, we follow Remark \ref{rem:cond_fun} and define ${\mu :=  \kappa\left( \PMP \right) - 1}$. For all the problems considered so far, the numerical results show that $\mu$ grows roughly linearly with respect to $p$.
This phenomenon can be clearly seen in Figure \ref{fig:Plot_sp}. The crucial consequence of this fact is that the number of PCG iterations is almost independent of $p$. This is confirmed by the results already shown in Tables  \ref{table:Kite_its}, \ref{table:Blade_its},  \ref{table:Hollow_its}, \ref{table:Disc1_its}  and \ref{table:Disc4_its}.

\renewcommand\arraystretch{1.1} 
\begin{table}
\centering                                         
	 \begin{tabular}{|c|c|c|c|c|c|}                   
		\hline                                           
		$n_{\mathrm{sub}}$ & $p=2$ & $p=3$ & $p=4$ & $p=5$ & $p=6$ \\ 
		 \hline                                            
		 16 & 1.692 & 1.861 & 2.018 & 2.173 & 2.330 \\ 
		 \hline                                            
		 32 & 1.696 & 1.866 & 2.024 & 2.177 & 2.330 \\ 
		 \hline                                            
		 64 & 1.699 & 1.869 & 2.028 & 2.182 & 2.334 \\ 
		 \hline                                            
		 128 & 1.700 & 1.871 & 2.029 & 2.184 & 2.336 \\
		\hline  
	\end{tabular} 
	\caption{Condition number of preconditioned mass matrix for holed plate.} \label{table:Hollow_cond} 
\end{table}   

\renewcommand\arraystretch{1.1} 
\begin{table}
	\centering                                  
	\begin{tabular}{|c|c|c|c|c|c|}                   
		\hline                                           
		$n_{\mathrm{sub}}$ & $p=2$ & $p=3$ & $p=4$ & $p=5$ & $p=6$ \\ 
		\hline                                                                      
		16 & 6 / 0.00181 & 7 / 0.00227 & 7 / 0.00248 & 7 / 0.00275 & 7 / 0.00311 \\ 
		\hline                                                                      
		32 & 6 / 0.00335 & 6 / 0.00370 & 6 / 0.00414 & 6 / 0.00466 & 6 / 0.00528 \\  
		\hline                                                                      
		64 & 5 / 0.00734 & 6 / 0.00968 & 6 / 0.0108\z & 6 / 0.0124\z & 6 / 0.0151\z \\    
		\hline                                                                      
		128 & 5 / 0.0250\z & 5 / 0.0287\z & 5 / 0.0343\z & 5 / 0.0405\z & 5 / 0.0488\z \\      
		\hline  
	\end{tabular}
	\caption{Iterations and time spent by PCG for holed plate.}
	\label{table:Hollow_its} 
\end{table}

\renewcommand\arraystretch{1.1} 
\begin{table}
\centering                                                                     
\begin{tabular}{|c|c|c|c|c|c|}                   
		\hline                                           
		$n_{\mathrm{sub}}$ & $p=2$ & $p=3$ & $p=4$ & $p=5$ & $p=6$ \\ 
		\hline                                            
		16 & 1.093 & 1.170 & 1.249 & 1.323 & 1.395 \\ 
		\hline                                            
		32 & 1.090 & 1.159 & 1.230 & 1.305 & 1.381 \\ 
		\hline                                            
		64 & 1.082 & 1.148 & 1.212 & 1.276 & 1.339 \\ 
		\hline                                            
		128 & 1.077 & 1.140 & 1.200 & 1.259 & 1.317 \\
		\hline  
	\end{tabular} 
	\caption{Condition number of preconditioned mass matrix for disc with one singularity.}
	\label{table:Disc1_cond} 
\end{table}   

\renewcommand\arraystretch{1.1} 
\begin{table}
	\centering                                  
	\begin{tabular}{|c|c|c|c|c|c|}                   
		\hline                                           
		$n_{\mathrm{sub}}$ & $p=2$ & $p=3$ & $p=4$ & $p=5$ & $p=6$ \\ 
		\hline                                                                     
		16 & 5 / 0.00156 & 5 / 0.00172 & 5 / 0.00192 & 6 / 0.00259 & 5 / 0.00259 \\
		\hline                                                                     
		32 & 4 / 0.00198 & 5 / 0.00274 & 5 / 0.00305 & 5 / 0.00355 & 5 / 0.00416 \\
		\hline                                                                     
		64 & 4 / 0.00435 & 4 / 0.00498 & 5 / 0.00693 & 5 / 0.00815 & 5 / 0.00951 \\
		\hline                                                                     
		128 & 4 / 0.0125\z & 4 / 0.0146\z & 4 / 0.0175\z & 4 / 0.0220\z & 5 / 0.0330\z \\  
		\hline        
	\end{tabular}
	\caption{Iterations and time spent by PCG for disc with one singularity.}
	\label{table:Disc1_its} 
\end{table}

\renewcommand\arraystretch{1.1} 
\begin{table}
\centering
\begin{tabular}{|c|c|c|c|c|c|}                   
		\hline                                           
		$n_{\mathrm{sub}}$ & $p=2$ & $p=3$ & $p=4$ & $p=5$ & $p=6$ \\ 
		\hline                                            
		16 & 1.167 & 1.252 & 1.350 & 1.459 & 1.575 \\ 
		\hline                                            
		32 & 1.161 & 1.241 & 1.341 & 1.450 & 1.564 \\ 
		\hline                                            
		64 & 1.158 & 1.237 & 1.338 & 1.447 & 1.559 \\ 
		\hline                                            
		128 & 1.156 & 1.236 & 1.336 & 1.444 & 1.556 \\
		\hline 
	\end{tabular} 
	\caption{Condition number of preconditioned mass matrix for disc with four singularities.}
	\label{table:Disc4_cond} 
\end{table}   

\renewcommand\arraystretch{1.1} 
\begin{table}
	\centering                                  
	\begin{tabular}{|c|c|c|c|c|c|}                   
		\hline                                           
		$n_{\mathrm{sub}}$ & $p=2$ & $p=3$ & $p=4$ & $p=5$ & $p=6$ \\ 
		\hline                                                                     
		16 & 5 / 0.00157 & 5 / 0.00172 & 6 / 0.00203 & 6 / 0.00229 & 6 / 0.00245 \\
		\hline                                                                     
		32 & 5 / 0.00282 & 5 / 0.00310 & 5 / 0.00329 & 5 / 0.00368 & 6 / 0.00478 \\ 
		\hline                                                                     
		64 & 4 / 0.00575 & 4 / 0.00648 & 5 / 0.00858 & 5 / 0.00983 & 5 / 0.0117\z \\ 
		\hline                                                                     
		128 & 4 / 0.0191\z & 4 / 0.0223\z & 4 / 0.0264\z & 4 / 0.0315\z & 4 / 0.0378\z \\    
		\hline        
	\end{tabular}
	\caption{Iterations and time spent by PCG for disc with four singularities.}
	\label{table:Disc4_its} 
\end{table}

We now compare our preconditioner $\P$ as defined in
\eqref{def:single_patch_prec} with the preconditioner proposed by  Chan
and Evans in
\cite{CHAN201822}, that we denote by $\PCE$.  This preconditioner has
some similarity with the one we propose and, moreover, is one of the
best performing to our knowledge. The  application of  $\PCE$
is, by definition,  a multiplication  by  
\begin{displaymath}
\PCE^{-1} = \widehat{\mathbf{M}} ^{-1}\mathbf{M}_{\jac{\vett[F]}^{-1}}\widehat{\mathbf{M}} ^{-1},
\end{displaymath}
where $ \mathbf{M}_{\jac{\vett[F]}^{-1}}$ is a weighted mass matrix  as
\eqref{eq:mass_weighted} with $\omega=\jac{\vett[F]}^{-1}$.  Table
\ref{table:comparison_k}   reports on the condition number of the
preconditioned mass matrix, by $\P$ and $\PCE $.  For a more in-depth
analysis of the efficiency of the two methods, we need to consider
that  one iteration of  $\PCE$-PCG costs roughly as two iterations
of $\P$-PCG.  This is because in the latter case the cost is concentrated in the matrix-vector product with $\mathbf{M}$ and the solution of a system with $\widehat{\mathbf{M}}$, while in the former case two such products (one with $\mathbf{M}$ and one with $\mathbf{M}_{\jac{\vett[F]}^{-1}}$) and two such solutions are needed. 

It is well-known that when Conjugate Gradient (CG)  is used to solve a linear system $\mathbf{A} {x} = \mathbf{b}$, with $\mathbf{A}$ symmetric and positive definite, it holds
$$ \frac{\left\| \mathbf{e}_k \right\|_\mathbf{A}}{\left\| \mathbf{e}_0 \right\|_\mathbf{A}} \leq  2 \left( \frac{\sqrt{\kappa\left(\mathbf{A}\right)} - 1}{\sqrt{\kappa\left(\mathbf{A}\right)} + 1} \right)^k, \qquad k = 1,2,\ldots, $$
where $\mathbf{e}_k$ is the error relative to the $k-$th iteration, and $ \left\|\mathbf{e}_k\right\|_\mathbf{A} := \sqrt{\mathbf{e}_k^T \mathbf{A} \mathbf{e}_k}$ for $k\geq 0$.
Thus, at each iteration of CG , the upper bound on the relative error is reduced by a factor 
\begin{equation}\label{eq:def_q}
	q\left(\mathbf{A}\right) :=\frac{\sqrt{\kappa\left(\mathbf{A}\right)} - 1}{\sqrt{\kappa\left(\mathbf{A}\right)} + 1} < 1.
\end{equation}
 In our case, $\mathbf{A}$ is the preconditioned mass matrix. Then, 
we use this principle in order to compare the effectiveness of $\P$ and $\PCE$. 
Since one iteration of $\PCE$-PCG costs twice as one
iterations of $\P$-PCG, we compare $q(\PMPCE)$ with the factor by
which the error bound is reduced after 2 iterations of $\P$-PCG, which
is $q(\PMP)^2$. The results are shown in Table
\ref{table:comparison}. In all cases, the bound-reducing factor is significantly small, confirming that both approaches lead to fast
solvers, with an advantage for   $\P$   in all the considered problems
and especially in the case of the singular parametrizations considered.

\renewcommand\arraystretch{1.1} 
\begin{table}
	\centering                                  
	\begin{tabular}{|c|c|c| }                
		\hline 	domain & $\kappa(\PMPCE)$& $\kappa(\PMP)$  \\ 
		\hline     	kite & $1.049$ & $1.157$ \\ 
		\hline 	blade & $1.202$ & $1.538$  \\ 
		\hline 	holed plate & $1.051$ & $1.216$ \\ 
		\hline 	disc (e) & $3.652$ & $1.395$ \\ 
		\hline 	disc (f) & $3.185$ & $1.575$ \\ 
		\hline
	\end{tabular}
\caption{ Condition number of preconditioned mass matrix for $n_{\mathrm{sub}}=16$ and $p=6$:
	comparison between $ \P$ and $ \PCE $. }
\label{table:comparison_k} 
\end{table}
\begin{table}
	\centering    
	\begin{tabular}{|c|c|c| }
		\hline 	domain & $q(\PMPCE)$& $q(\PMP)^2$  \\ 
		\hline     	kite & $1.20 \cdot 10^{-2}$ & $1.33 \cdot 10^{-3}$ \\ 
		\hline 	blade & $4.60 \cdot 10^{-2}$ & $1.15 \cdot 10^{-2}$  \\ 
		\hline 	holed plate & {\Rd$1.24 \cdot 10^{-2}$} & {\Rd$2.39 \cdot 10^{-3}$} \\ 
		\hline 	disc (e) & $3.13 \cdot 10^{-1}$ & $6.89 \cdot 10^{-3}$ \\ 
		\hline 	disc (f) & $2.82 \cdot 10^{-1}$ & $1.28 \cdot 10^{-2}$ \\ 
		\hline
	\end{tabular}
	\caption{\Rd 
Error reduction factors relative to one iteration of $\PCE$-PCG (left
column) and two  iterations of $\P$-PCG (right column), having a similar
computation cost. The factors refer to the case  $n_{\mathrm{sub}}=16$ and $p=6$.
}	\label{table:comparison} 
\end{table}

\subsection{Multipatch domains}

Finally, in order to evaluate the performance of our Additive
Schwarz preconditioner, we consider three domains: a multipatch five-pointed star
(Figure \ref{fig:Star_mp}), a multipatch disc (Figure
\ref{fig:Disc_mp}) and a multipatch fan (Figure \ref{fig:Fan_mp}),
obtained by gluing together 7 blade-shaped patches like the one represented in
Figure \ref{fig:Blade}. 

As in the single patch case, we compare the condition number of the original mass matrix (Tables \ref{table:Star_mp_cond_ump} and \ref{table:Disc_mp_cond_ump}) with that of the preconditioned one (Tables \ref{table:Star_mp_cond} and \ref{table:Disc_mp_cond}). In all cases, the preconditioner greatly reduces the condition number of the matrix, robustly with respect to $h$.  Moreover, the growth of the condition number with respect to the spline degree $p$ seems to be linear (see Figure \ref{fig:Plot_mp}). This is reflected also in the number of iterations needed by PCG to reach the given tolerance, see Tables \ref{table:Disc_mp_its} and \ref{table:Fan_mp_its}.


\begin{table}
	\centering                                 
	\begin{tabular}{|c|c|c|c|c|c|}                   
		\hline                                           
		$n_{\mathrm{sub}}$ & $p=2$ & $p=3$ & $p=4$ & $p=5$ & $p=6$ \\ 
		\hline                                                            
		16 & $1.326 \cdot 10^2$ & $9.994 \cdot 10^2$ & $6.848 \cdot 10^3$ & $4.611 \cdot 10^4$ & $3.160 \cdot 10^5$ \\ 
		\hline                                                            
		32 & $1.503 \cdot 10^2$ & $1.156 \cdot 10^3$ & $8.104 \cdot 10^3$ & $5.565 \cdot 10^4$ & $3.845 \cdot 10^5$ \\
		\hline                                                            
		64 & $1.618 \cdot 10^2$ & $1.258 \cdot 10^3$ & $8.935 \cdot 10^3$ & $6.217 \cdot 10^4$ & $4.351 \cdot 10^5$ \\
		\hline     	
	\end{tabular} 
	\caption{Condition number of mass matrix for multipatch star.} \label{table:Star_mp_cond_ump} 
\end{table} 

\begin{table}
\centering                                 
\begin{tabular}{|c|c|c|c|c|c|}                   
\hline                                           
$n_{\mathrm{sub}}$ & $p=2$ & $p=3$ & $p=4$ & $p=5$ & $p=6$ \\ 
\hline                                                
16 & 39.69 & 48.05 & 56.17 & 64.03 & 71.62 \\
\hline                                                
32 & 39.80 & 48.23 & 56.42 & 64.32 & 71.95 \\
\hline                                                
64 & 39.86 & 48.33 & 56.55 & 64.48 & 72.13 \\
\hline      	
\end{tabular} 
\caption{Condition number of preconditioned mass matrix for multipatch star.} \label{table:Star_mp_cond} 
\end{table} 

\begin{table}
	                           \centering                                 
\begin{tabular}{|c|c|c|c|c|c|}                   
\hline                                           
$n_{\mathrm{sub}}$ & $p=2$ & $p=3$ & $p=4$ & $p=5$ & $p=6$ \\ 
\hline                                                                        
16 & 13 / 0.0111 & 14 / 0.0134 & 14 / 0.0145 & 15 / 0.0175 & 15 / 0.0198 \\   
\hline                                                                        
32 & 12 / 0.0251 & 12 / 0.0276 & 13 / 0.0320 & 13 / 0.0375 & 14 / 0.0488 \\    
\hline                                                                        
64 & 10 / 0.0647 & 12 / 0.0892 & 12 / 0.102\z & 12 / 0.119\z & 12 / 0.141\z \\      
\hline                                                                        
128 & \z9 / 0.221\z & 11 / 0.303\z & 11 / 0.348\z & 11 / 0.393\z & 12 / 0.502\z \\   
\hline
\end{tabular} 
\caption{Iterations and time spent by PCG for multipatch star.}
\label{table:Star_mp_its} 
\end{table}


\begin{table}
	\centering                                 
	\begin{tabular}{|c|c|c|c|c|c|}                   
		\hline                                           
		$n_{\mathrm{sub}}$ & $p=2$ & $p=3$ & $p=4$ & $p=5$ & $p=6$ \\ 
		\hline                                                              
		16 & $2.098 \cdot 10^2$ & $1.550 \cdot 10^3$ & $1.049 \cdot 10^4$ & $7.007 \cdot 10^4$ & $4.761 \cdot 10^5$ \\ 
		\hline                                                              
		32 & $2.585 \cdot 10^2$ & $1.970 \cdot 10^3$ & $1.374 \cdot 10^4$ & $9.360 \cdot 10^4$ & $6.399 \cdot 10^5$ \\ 
		\hline                                                              
		64 & $2.949 \cdot 10^2$ & $2.291 \cdot 10^3$ & $1.637 \cdot 10^4$ & $1.143 \cdot 10^5$ & $7.994 \cdot 10^5$ \\
		\hline   	
	\end{tabular} 
	\caption{Condition number of mass matrix for multipatch disc.} \label{table:Disc_mp_cond_ump} 
\end{table}

\begin{table}
\centering                                 
\begin{tabular}{|c|c|c|c|c|c|}                   
\hline                                           
$n_{\mathrm{sub}}$ & $p=2$ & $p=3$ & $p=4$ & $p=5$ & $p=6$ \\ 
\hline                                                
16 & 13.88 & 16.02 & 18.03 & 19.92 & 21.70 \\
\hline                                                
32 & 13.99 & 16.16 & 18.18 & 20.08 & 21.87 \\
\hline                                                
64 & 14.06 & 16.24 & 18.28 & 20.18 & 21.98 \\
\hline   	
\end{tabular} 
\caption{Condition number of preconditioned mass matrix for multipatch disc.} \label{table:Disc_mp_cond} 
\end{table} 

\begin{table}
	\centering                                 
\begin{tabular}{|c|c|c|c|c|c|}                   
\hline                                           
$n_{\mathrm{sub}}$ & $p=2$ & $p=3$ & $p=4$ & $p=5$ & $p=6$ \\ 
\hline                                                                       
16 & 14 / 0.0125 & 15 / 0.0147 & 17 / 0.0180 & 17 / 0.0203 & 18 / 0.0240 \\    
\hline                                                                       
32 & 14 / 0.0310 & 15 / 0.0365 & 16 / 0.0424 & 17 / 0.0513 & 17 / 0.0617 \\   
\hline                                                                       
64 & 14 / 0.0995 & 14 / 0.114\z & 16 / 0.145\z & 16 / 0.170\z & 16 / 0.196\z \\       
\hline                                                                       
128 & 14 / 0.380\z & 14 / 0.421\z & 15 / 0.496\z & 16 / 0.606\z & 16 / 0.704\z \\ 
\hline       
\end{tabular} 
\caption{Iterations and time spent by PCG for multipatch disc.}
\label{table:Disc_mp_its} 
\end{table}


\begin{table}
\centering                                  
\begin{tabular}{|c|c|c|c|c|c|}                   
\hline                                           
$n_{\mathrm{sub}}$ & $p=2$ & $p=3$ & $p=4$ & $p=5$ & $p=6$ \\ 
\hline                                                              
16 & 12 / 0.205 & 13 / \z0.436 & 13 / \z0.844 & 14 / 1.96 & 15 / \z3.76 \\
\hline                                                              
32 & 10 / 1.11\z & 10 / \z2.04\z & 12 / \z4.50\z & 12 / 9.02 & 12 / 16.0\z \\  
\hline
64 & \z9 / 7.70\z & \z9 / 13.2\z\z & 10 / 27.7\z\z & * & * \\
\hline                                                   
\end{tabular}
\caption{Iterations and time spent by PCG for multipatch fan. In the cases denoted by ``*'', we were not able to assemble the mass matrix due to memory limitations.}
\label{table:Fan_mp_its} 
\end{table}

\begin{figure}
	\begin{subfigure}{0.49\textwidth}
		\includegraphics[width=\textwidth]{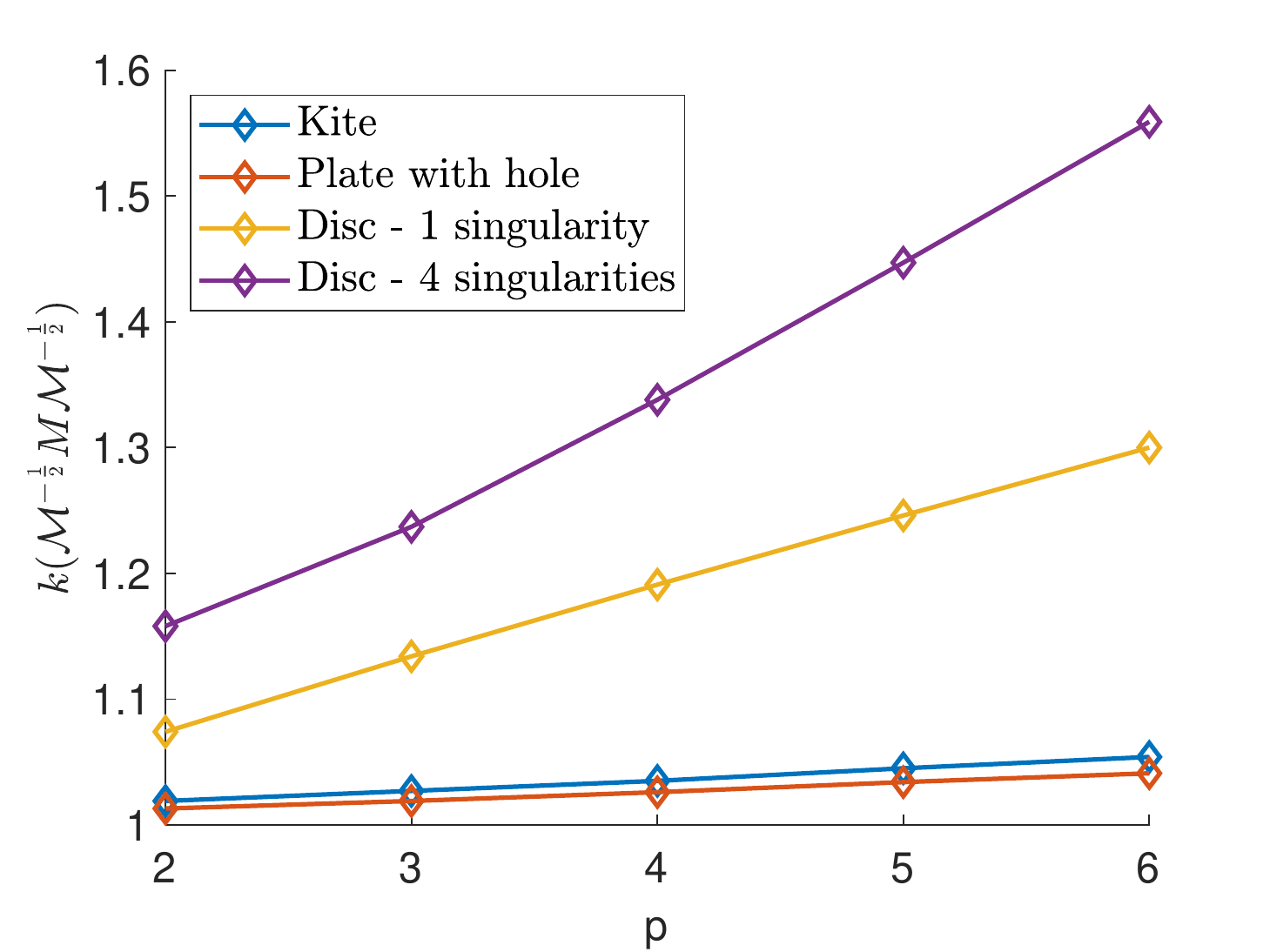}
		\caption{Single patch domains}
		\label{fig:Plot_sp}
	\end{subfigure}
	\begin{subfigure}{0.49\textwidth}
		\includegraphics[width=\textwidth]{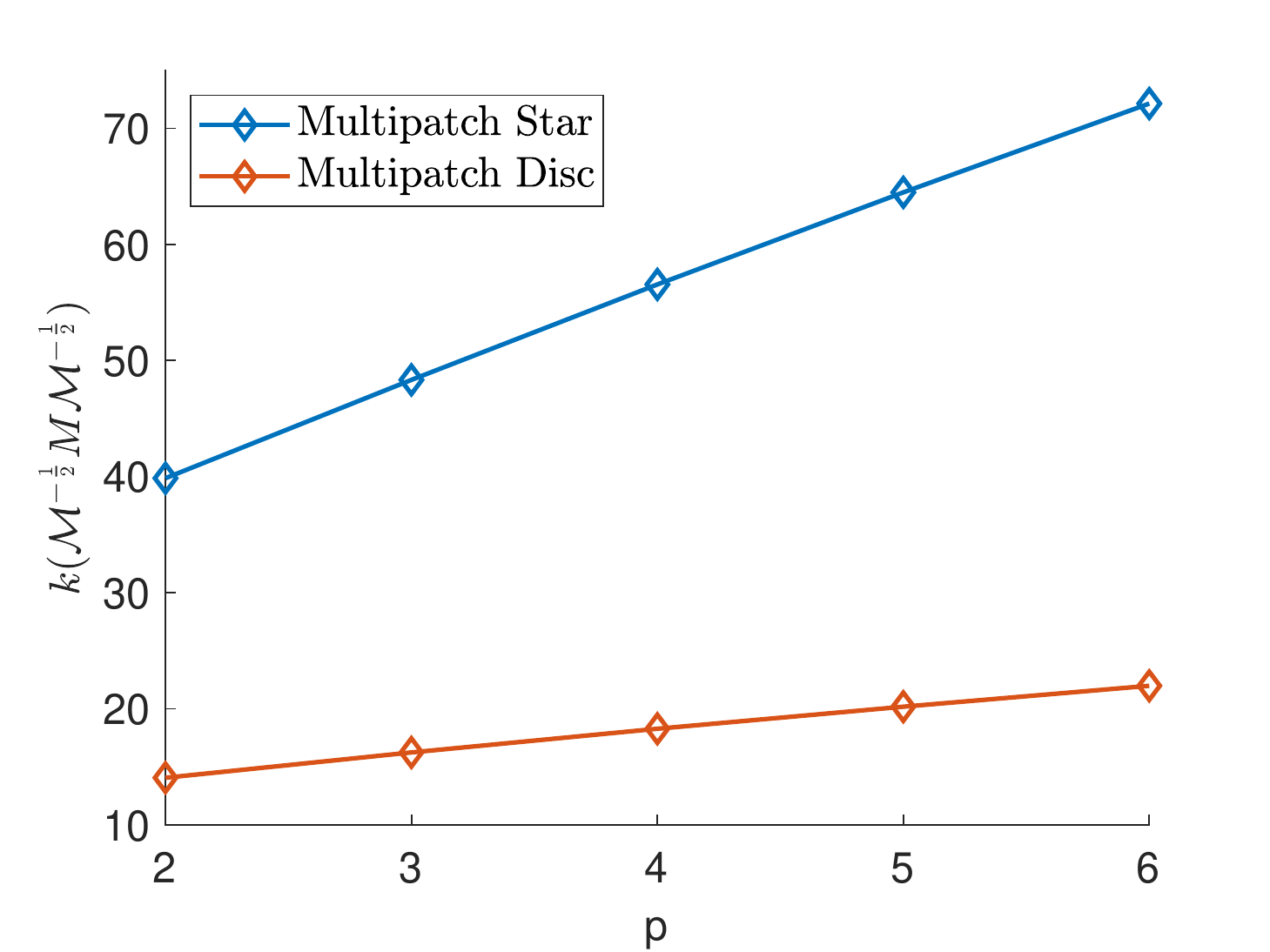}
		\caption{Multipatch domains}
		\label{fig:Plot_mp}
	\end{subfigure}
	\caption{Condition number of preconditioned mass matrix ($n_{\mathrm{sub}}=64$).}
	\label{fig:Plot_p}
\end{figure}


\section{Conclusions}\label{sec:conclusion}
In this work, we have presented a simple and  efficient preconditioner
for mass matrices arising in isogeometric analysis. The main idea for
the single patch case is to exploit the Kronecker product structure of
parametric mass matrix on the reference domain, combined with a
diagonal scaling to correctly  incorporate the effect of the geometry
parametrization. In order to deal with 
multipatch domains, we have used the single patch strategy in an
Additive Schwarz preconditioner.  The preconditioner  has an application
cost  of $O(p\Ndof)$ FLOPs, and  is  well suited for
parallelization. We have proved  that the single-patch
preconditioner converges, as  the
mesh-size $h$ goes to $0$, to the exact mass, and that robustness  with respect to
$h$ is preserved in the multipatch case.  Numerical tests reflect
the theoretical results and
show a very good behaviour also with respect to the spline
degree $p$.

\section*{Acknowledgements}
The authors were partially supported by the European Research Council
through the FP7 Ideas Consolidator Grant HIGEOM n.616563,  and
by the Italian Ministry of Education, University and Research (MIUR)
through the  ``Dipartimenti di Eccellenza Program (2018-2022) -
Dept. of Mathematics, University of Pavia''.  This support are
gratefully acknowledged. The authors
are members of the Gruppo Nazionale Calcolo Scientifico-Istituto
Nazionale di Alta Matematica (GNCS-INDAM).

\bibliography{biblio_mass}

\end{document}